\documentclass[12pt]{amsart}
\usepackage{graphicx}
\usepackage{tikz}
\usetikzlibrary{decorations.markings}
\usepackage{subfig}
\usepackage{enumerate}
\usepackage{pict2e}
\newtheorem{thm}{Theorem}[section]

\newtheorem{lem}[thm]{Lemma}

\theoremstyle{remark}
\newtheorem{remark}{Remark}[section]
\newtheorem{example}{Example}[section]

\numberwithin{equation}{section}
\numberwithin{figure}{section}

\newcommand{\thmref}[1]{Theorem~\ref{#1}}

\newcommand{\bg}{\begin}

\def\jp12{{j+1/2}}
\def\jm12{{j-1/2}}

\def\R {\mathbb{R}}

\def \e {\epsilon}

\def \RR{{\rm I\kern -1.6pt{\rm R}}}

\newcommand{\ds}{\displaystyle}
\newcommand{\sign}{\mathop{\rm sign}\nolimits}

\usepackage{pict2e} 

\newcommand{\drm}{\,{\rm d}}

\newcommand*\circled[1]{\tikz[baseline=(char.base)]{
            \node[shape=circle,draw,inner sep=1pt] (char) {#1};}} 

\usetikzlibrary{decorations.pathreplacing,decorations.markings}
\tikzset{
  on each segment/.style={
    decorate,
    decoration={
      show path construction,
      moveto code={},
      lineto code={
        \path [#1]
        (\tikzinputsegmentfirst) -- (\tikzinputsegmentlast);
      },
      curveto code={
        \path [#1] (\tikzinputsegmentfirst)
        .. controls
        (\tikzinputsegmentsupporta) and (\tikzinputsegmentsupportb)
        ..
        (\tikzinputsegmentlast);
      },
      closepath code={
        \path [#1]
        (\tikzinputsegmentfirst) -- (\tikzinputsegmentlast);
      },
    },
  },
  mid arrow/.style={postaction={decorate,decoration={
        markings,
        mark=at position .6 with {\arrow[#1]{stealth}}
      }}},
}

\begin{document}
\title{A model for traffic flow encoding hysteresis}

\author{Andrea Corli}
\address{Department of Mathematics and Computer Science, University of Ferrara\\
44121 Ferrara, Italy}
\email{andrea.corli@unife.it}

\author{Haitao Fan}
\address{Department of Mathematics,
Georgetown University\\
Washington, DC 20057}
\email{fan@math.georgetown.edu}
\date{\today}

\bg{abstract}
A simple macroscopic model for the vehicular traffic flow with hysteresis is proposed.
The model includes drivers' hysteresis behavior into the classical Lighthill-Whitham-Richard (LWR) model.
One novelty of the model is how the hysteresis is modeled based on hysteresis behavior reported in literature.
By comparing wave patterns producible by the classical LWR model versus this model, the effect of hysteresis can be revealed.
For this purpose, wave pattern analysis is carried out by construction of viscous waves and Riemann solvers of the model.
Examples of solutions for initial data corresponding to some real world setups or experimental observations are constructed.
Many of them qualitatively
agree with experimental observations and intuition.
These result shows that the model not only keeps the successful part of the LWR model,
but also demonstrates many wave patterns observed real world traffic that LWR fails to exhibit,
such as stop-and-go waves and many other traffic patterns. A mechanism for formation and decay of stop-and-go patterns and traffic jams is revealed.
Furthermore, acceleration shocks in congested zone are predicted, which can serve as fronts of stop-and-go patterns. 
\bigskip


\end{abstract}

\maketitle

\baselineskip 18pt

\section{Introduction}
The simplest equation providing a macroscopic description of a traffic flow is the famous Lighthill-Whitham-Richards equation \cite{Lighthill-Whitham, Richards}
\begin{equation}\label{e:LWR}
\rho_t+\left(\rho v\right)_x=0,
\end{equation}
where $\rho\in[0,\rho_{\rm max}]$ is the normalized car density and the velocity $v$ is assumed to be a given function of the density $\rho$. Usually $v$ is a decreasing function and $v(\rho_{\rm max})=0$. It is clear that the complexity of traffic flows cannot be fully described by \eqref{e:LWR}; we refer to \cite{CLS, Helbing} for two dated survey articles, which however still provide a lot of insights to this modeling. The agreement and disagreement between LWR model and real-world traffic observations are summarized in \cite{Nagel-Nelson2005}. Major drawbacks of LWR model are that it disallows both stop-and-go waves and car platoons traveling at uniform speed without uniform spacing, among many other wave patterns usually observed in real traffic.

Macroscopic traffic models of conservative form cannot produce stop-and-go patterns when both front of a stop-and-go wave are sharp. In a recent work \cite{Seibold2013}, the authors constructed jamiton solutions, resembling stop-and-go waves, for Payne-Whitham (PW) model, Aw-Rascle (AR) model \cite{Aw-Rascle} and inhomogeneous Aw-Rascle-Zhang (ARZ) model \cite{Zhang2002}.
However, the fixed length and a smooth acceleration part of a jamiton is contrary to some observations of a stop-and-go structure in real-world traffic, see e.g. Fig. 3 of \cite{Kaufmann_et_al2017}.

In the existing literature, drivers' hysteresis behavior is reported. Hysteresis loops are believed to be related to stop-and-go waves and many other traffic phenomena. But hysteresis traffic is never fully modeled, at least in macroscopic models, and the above belief is not confirmed nor rejected deductively.

In this paper we shall investigate traffic with hysteresis by building hysteresis effect into the LWR model \eqref{e:LWR}.
We shall investigate and list possible wave patterns of this model and compare them with those allowable by the LWR model.
Then the differences are attributable to hysteresis.

The occurrence of different acceleration and deceleration regimes for congested traffic flows was first pointed out in \cite{Newell1962}; more precisely, in the case of congested flow, the deceleration branch was required to lie above the acceleration branch in the density-flow diagram (the so-called {\em fundamental diagram}), as confirmed by empirical data \cite{Edie1965, Foote1965}. A reason for this asymmetry is that drivers keep large spacing in acceleration and short spacing in deceleration. It was later recognized  \cite{Newell1965} that the occurrence of these different regimes is the starting point for an hysteresis-like interpretation, and hysteresis loops were confirmed experimentally \cite{Treiterer-Myers}. Some microscopic models for traffic flows showing hysteretic behaviors were reviewed in \cite{CLS}, and updated experimental evidence was provided there.

To the best of our knowledge, the first paper proposing a macroscopic {\em modeling} of hysteresis in traffic flows is \cite{Zhang1999}. By means of a careful analysis of the experimental data, the author deduces that the acceleration and deceleration trajectories $\left(\rho(x(t)),v(x(t))\right)$ {\em meet}, and possibly more than once; here $x(t)$ is the car path as a function of time. Then, he builds up a complex model with three different phases, where hysteresis in then {\em deduced} as a consequence of phase changes. We refer to Figure \ref{f:8} for an illustration of an hysteresis loop, see \cite[Fig. 3]{Treiterer-Myers} and \cite[Fig. 9]{Zhang1999}. On a similar basis, further considerations on hysteresis loops are contained in \cite{Yeo2008, Yeo-Skabardonis}, where five traffic phases are identified, including a stationary and a coasting phase. We also refer to \cite{Zhang-Kim} for another modeling with a coasting phase.

\begin{figure}[hb]
\begin{tikzpicture}[>=stealth, scale=0.6]


\draw[->] (0,0) -- (12,0) node[below=0.2cm, left] {$\rho$} coordinate (x axis);
\draw[->] (0,0) -- (0,8) node[left] {$v$} coordinate (y axis);

\draw[thick]  [rounded corners=0.3cm ] (10,1) .. controls (2,1.5) and (4,3)  .. node[near start, below, sloped]{acceleration} node[near start, above]{$\ell_+$} (3,7)
-- node[midway, below=0.7cm,right=-0.1cm]{$\ell_-$} (1,7.2) [sharp corners] .. controls (1.3,4) and (6,3) .. (10,1.2) node[near end, above, sloped]{deceleration} .. controls (10.3,1.05) and (10.5,1) .. (10,1);

\draw[very thick,->] (6.05,1.5) -- (6,1.52);
\draw[very thick,->] (6,2.8) -- (6.05,2.78) ;
\draw[very thick,->] (3.3,5.5) -- (3.29,5.55) ;
\draw[very thick,->] (1.58,5.5) -- (1.63,5.45) ;

\end{tikzpicture}
\caption{\label{f:8}{An figure eight-loop as in \cite[Fig. 3]{Treiterer-Myers} and \cite[Fig. 9]{Zhang1999} with a positive subloop $\ell_+$ and a negative subloop $\ell_-$.}}
\end{figure}

Recent analysis of experimental data \cite{Laval2011, Laval-Leclercq, SZHW}, though without proposing a precise model, attribute more importance to the heterogeneity of drivers to justify cluster formations. Here heterogeneity refers to the existence of aggressive and timid drivers. According to this interpretation, two kinds of hysteresis loops have been observed in experimental data. Loops are called {\em positive} if the acceleration branch stays {\em below} the deceleration branch and the loop is traveled clockwise in the fundamental diagram; they are called {\em negative} if the acceleration branch stays {\em above} the deceleration branch and the loop is traveled counterclockwise \cite{AVL, Laval2011}. Notice that Figure \ref{f:8} shows the possibility of both loops. Experimental measurements of hysteresis can be found in \cite{SZHW, Yeo-Skabardonis} and positive loops are observed to be more frequent than negative loops. We also refer to \cite{Deng-Zhang} for a modeling including relaxation and anticipation parameters, through which one can reproduce positive, negative, and double hysteresis loops, as well as aggressive and timid driving behavior.

Other papers dealing with hysteresis in traffic flows are \cite{ABRGF}, where loops are observed from experimental data and a microscopic stochastic model is proposed; \cite{BABW}, where a phase-transition macroscopic model is studied; \cite{Blank2005, Blank2008} for a deterministic lattice model.

In the macroscopic fluid-dynamics models quoted above, stop-and-go wave patterns and related hysteresis-like phenomena observed in real traffic flows are almost never {\em explicitly and rigorously} deduced from the model.
Another observed traffic pattern is the {\em car train}: a sequence of vehicles traveling at the same speed, but with possibly different spacing between pairs of adjacent vehicles. This pattern is not allowed by macroscopic models as \eqref{e:LWR}, where $v=v(\rho)$ is prescribed. On the other hand, microscopic models tracing each individual vehicle are flexible enough to include various drivers' behaviors for each vehicle, and hence can produce many observed wave patterns; however, they are hard to analyze due to the size of the systems.

\smallskip

The main purpose of this paper is to construct a simple model of vehicular traffic flow in one lane, which not only includes the successful part of Lighthill-Whitham-Richards-type models, but can also
show explicitly the relation between hysteresis phenomena and the formation of stop-and-go waves, the existence of stationary cars train with varying spacings, the compaction of car trains due to speed variations, among other things. Our model is somewhat similar to analogous models dealing with hysteresis phenomena in oil recovery \cite{Corli-FanPM, Plohr-Marchesin-Bedrikovetsky-Krause}.

This paper is organized as follows. In \S2, a simple model of vehicular traffic flow in one lane with hysteresis is constructed in Eulerian coordinates.
In \S3, we translate the model into Lagrangian coordinates, which give easy access to drivers' point of view, especially because hysteresis loops are related to vehicle trajectories \cite{Zhang1999}. In the rest of the paper we always use the Lagrangian coordinates.
In \S4, all basic waves in the model, such as shock and rarefaction waves, are found and their existence is proved.
Unusual waves are the stationary shocks (in Lagrangian coordinates).
Such shocks can be argued intuitively later in \S6 as a result of drivers' rational behavior.
 Experimental data fully confirm the occurrence of these waves,
see \cite[page 210]{CLS} and \cite{Yeo2008, Yeo-Skabardonis}.
In \S5, Riemann solvers are constructed. Examples are given to show that Riemann solvers are not unique for some Riemann data.
In \S6, many examples of solutions of the model with various initial values corresponding interesting situations are constructed.
Examples of solutions exhibiting stop-and-go waves and stationary car trains with uneven spacings are presented.
Some examples show the formation and decay of stop-and-go waves.
In particular, one example investigates the formation of stop-and-go patterns on a circular track, echoing experimental observations reported in  \cite{Sugiyama2004}.
Other solutions show several interesting traffic patterns including the compaction of car trains due to speed variations.
We further investigate the effect of nonuniqueness of Riemann solvers.
In particular, one example shows the instability of a uniform car train in the congested zone can arise due to such nonuniqueness. In \S7 we conclude this paper.

\section{The model in Eulerian coordinates}\label{s:Euler}
Two different phases have been observed in one-lane traffic flows \cite{Hall-Allen-Gunter, Kerner}. One phase occurs when the vehicle density $\rho$ is low; the traffic is then said to be in {\em the free zone}. The other phase takes place when the density $\rho$ exceeds some critical value $\rho_c$; then the traffic is in {\em the congested zone}. The density-velocity relations are different in these two zones. Moreover, it has been recognized since \cite{Newell1962} that trajectories corresponding to accelerations and decelerations describe different curves in the fundamental diagram. By a careful analysis of the data in \cite{Treiterer-Myers}, it was deduced \cite{Zhang1999} that in the free zone the acceleration curve lies above the deceleration curve, and conversely in the congested zone, see Figure \ref{fig:vu}. The whole modeling in \cite{Zhang1999} is rigorous and based on few simple and reasonable assumptions.

Based on \cite{Newell1962, Treiterer-Myers, Zhang1999} and recalling \cite{Corli-FanPM} for a similar modeling in a different topic, we describe the behavior of a typical driver as follows. Given the vehicle density $\rho$ in front of the driver, the driver's speed takes values in a range between $v^A(\rho)$ and $v^D(\rho)$, depending on the driver's mode (or mood). The curves
$v=v^A(\rho)$ and $v=v^D(\rho)$ are called the acceleration and the deceleration curves, respectively. To be more specific, let us follow a typical rational driver in {\em congested} zone;
we refer to Figure \ref{f:vv}.

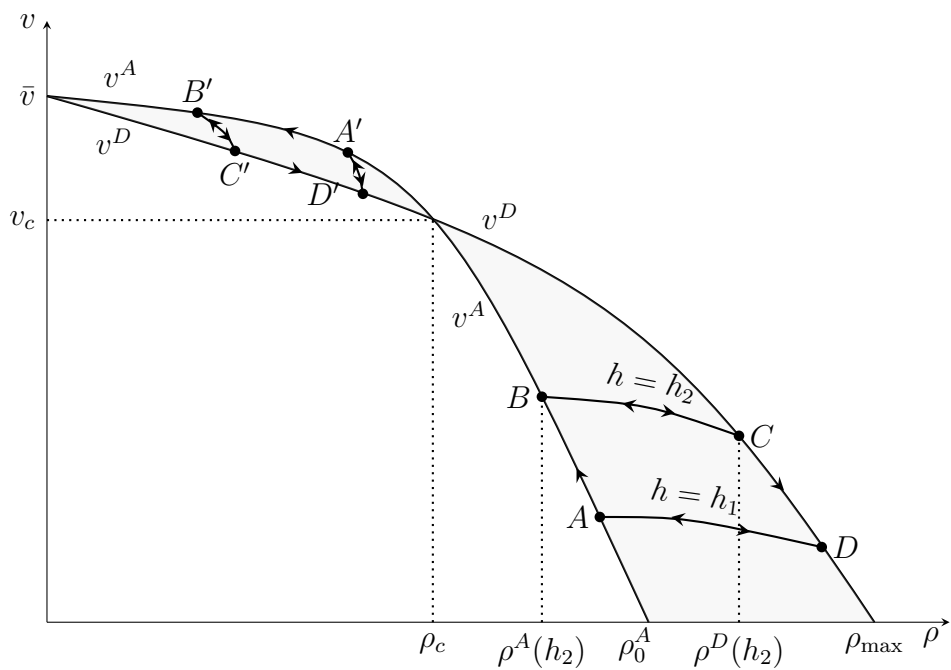
\begin{figure}[htbp]
\begin{tikzpicture}[>=stealth],scale=0.9]


\draw[->] (0,0) -- (12,0) node[below=0.2cm, left] {$\rho$} coordinate (x axis);
\draw[->] (0,0) -- (0,8) node[left] {$v$} coordinate (y axis);

\draw[thick] (0,7) node[below=0.6cm, right=0.5cm]{$v^D$} node[left]{$\bar v$} .. controls (7,5) and (8,4.5) .. (11,0)  node[below]{$\rho_{\rm max}$};
%
\draw[thick] (8,0) node[below=0.3cm, left=-0.2cm]{$\rho_0^A$} .. controls (5,6.5)  and (5,6.5) .. (0,7) node[above=0.3cm, right=0.6cm]{$v^A$};
%
\filldraw[fill=gray!20!white, draw=black, nearly transparent] (0,7) .. controls (7,5) and (8,4.5) ..  (11,0) -- (8,0) .. controls (5,6.5) and (5,6.5) .. (0,7) -- cycle;
%
\draw[dotted, thick] (5.13,0) node[below]{$\rho_c$} -- (5.13,5.4) 
node[below=1.3cm, right=0.1cm]{$v^A$}
node[right=0.5cm]{$v^D$};
\draw[dotted,thick] (0,5.35) node[left]{$v_c$} -- (5.2,5.35);


\fill (7.35,1.4) circle (2pt) node[left]{$A$};
\fill (6.58,3) circle (2pt) node[left]{$B$};
\fill (9.2,2.48) circle (2pt) node[right]{$C$};
\fill (10.3,1) circle (2pt) node[right]{$D$};

\draw[thick, dotted] (6.58,0) node[below]{$\rho^A(h_2)$} -- (6.58,3);
\draw[thick, dotted] (9.2,0) node[below]{$\rho^D(h_2)$} -- (9.2,2.48);

\draw[thick] (6.58,3)  .. controls (8,2.9) .. node[midway,sloped,above] {$h=h_2$} (9.2,2.48);
\draw[thick] (7.35,1.4)  .. controls (8.5,1.4) .. node[midway,sloped,above] {$h=h_1$} (10.3,1);

\draw[very thick,->] (7.065,2) -- (7.04,2.05);
\draw[very thick,->] (8.3,2.79) -- (8.35,2.78);
\draw[very thick,->] (7.7,2.9) -- (7.65,2.91);
\draw[very thick,->] (9.75,1.8) -- (9.8,1.75);
\draw[very thick,->] (9.3,1.21) -- (9.35,1.2);
\draw[very thick,<-] (8.3,1.4) -- (8.35,1.39);

\fill (4,6.25) circle (2pt) node[above]{$A'$};
\fill (2,6.78) circle (2pt) node[above]{$B'$};
\fill (2.5,6.27) circle (2pt) node[below]{$C'$};
\fill (4.2,5.7) circle (2pt) node[left=0.152cm]{$D'$};

\draw[thick] (2,6.78)  .. controls (2.4,6.5) ..  (2.5,6.27);
\draw[thick] (4,6.25)  .. controls (4.15,6) .. (4.2,5.7);

\draw[very thick,->] (3.2,6.55) -- (3.15,6.57);
\draw[very thick,->] (2.25,6.6) -- (2.15,6.7);
\draw[very thick,->] (2.4,6.45) -- (2.45,6.4);
\draw[very thick,->] (3.3,6.02) -- (3.4,5.98);
\draw[very thick,->] (4.12,6.05) -- (4.07,6.15);
\draw[very thick,->] (4.17,5.9) -- (4.2,5.8);

\end{tikzpicture}
\caption{\label{fig:vu}{The functions $v^A(\rho)$ and $v^D(\rho)$ and two scanning curves with $h_1<h_2$. The hysteresis loop in congested zone $ABCDA$ is clockwise, while that in free zone, $A'B'C'D'A'$, is counterclockwise}. Arrows indicate the directions in which  a car's state can move. }
\label{f:vv}
\end{figure}

{\em Acceleration.} Suppose that the driver's state at time $t=t_a$ is at point $A$. If the spacing, defined as $1/\rho$, in front of him/her 
 is increasing as $t$ increases, then the driver will increase his/her speed along the curve $v=v^A(\rho)$. As long as vehicles in front of him/her obey the same law, he/she will drive at that speed too. In fact, traveling faster than $v=v^A(\rho)$ would result in a forced deceleration soon after the initial acceleration that costs more in fuel, and end up at the same speed as the car in front of his/her but with a shorter, and hence less safe and less comfortable, spacing in front. Thus, he/she follows the curve $v=v^A(\rho)$ as he/she accelerates.

{\em Scanning to deceleration.} After some time, at $t = t_b$, the driver's state is at the point $B$. Suppose now the spacing ahead of his/her is decreasing as time increases; this forces him to decelerate. Hoping the space shortening is just temporary, and out of mental inertia as well as being reluctant to loose momentum, he/she will not decelerate along the curve $v=v^A(\rho)$. Rather, he/she will decelerate along a curve with a faster speed for the same $\rho$, shown in Figure \ref{fig:vu} as the curve $BC$. If the spacing improves after the initial deceleration, then he/she can accelerate along the curve $BC$ back to point $B$. Had he/she followed the curve $v=v^A(\rho)$ for deceleration, he/she would go through much harder decelerations and accelerations, costing more in time, fuel and vehicle maintenance.
Borrowing a terminology used for modeling hysteresis in oil recovery \cite{Corli-FanPM, Plohr-Marchesin-Bedrikovetsky-Krause}, we call the $BC$ curve a {\it scanning curve} and  {\it scanning mode} the corresponding mode of the driver.

We assume, for simplicity, that a typical driver in scanning mode accelerates and decelerates along the {\em same} scanning curve.

{\em Deceleration.} Suppose that the spacing in front of the driver keeps shortening; then he/she has to decelerate along the curve $BC$, and eventually hits the point $C$ at time $t=t_c$.
At $C$, the spacing is the tightest possible for the speed. If further deceleration has to be done, he/she must decelerate harder along the deceleration curve $v=v^D(\rho)$ for safety.

{\em Scanning to acceleration.} Assume that later, at $t=t_d$, the driver's state is at point $D$,
and the spacing in front of him is increasing. Then
he/she exits the deceleration mode, gets into scanning mode and accelerates in response. However, with bad memories of traffic jams and being reluctant to invest more fuel in an unsure recovery, he/she drives slower than what $v=v^D(\rho)$ allows, along the scanning curve $DA$, so that he/she can decelerate easily back should the euphoria fades. Next, suppose that the spacing keeps improving to point $A$. Further spacing improvement after that will send him to acceleration mode and further acceleration must follow the acceleration curve $v=v^A(\rho)$.

\smallskip

The loop $ABCDA$ is called a {\em hysteresis loop}; analogous loops occur in the {\em free} zone, see the loop $A'B'C'D'A'$ in Figure \ref{fig:vu}. Notice that loops in congested zone and free zone rotate in opposite directions. It is not clear whether these loops coincide with those described in \cite{AVL, Laval2011, SZHW}.
Their interpretation in \cite{AVL, Laval2011, SZHW} is based on the heterogeneity of the drivers and not on the presence of the acceleration and deceleration curves. It is likely that these loops simply occur in different traffic regimes and that the introduction in our model of a further variable accounting for heterogeneity could possibly explain both kinds of loops. We leave this subject for future investigations.

In summary, the driver's mode can be classified  as follows:
\begin{equation}\label{e:imb-dr-begin}
\begin{array}{rl}
\hbox{ Acceleration mode: } & \rho_t(x,\cdot)<0\text{ and } v=v^A(\rho);
\\[2mm]
\hbox{ Deceleration mode: } & \rho_t(x,\cdot)>0 \text{ and } v=v^D(\rho);
\\[2mm]
\hbox{ Scanning mode: } & v \text{ is between } v^A(\rho)
 \text{ and } v=v^D(\rho), \text{ and }
 \\ & \text{neither in Acceleration nor Deceleration mode.}
\end{array}
\end{equation}

\smallskip

After this illustration of the model, we now provide precise details. We are given two smooth positive functions $v^A(\rho)$ and $v^D(\rho)$, with $v^A(0)=v^D(0)=\bar v$ and $v^A(\rho_0^A)=v^D(\rho_{\rm max})=0$. Here $\bar v>0$ is the maximum speed and $0<\rho_0^A<\rho_{\rm max}$, see \cite{Newell1962, Yeo-Skabardonis, Zhang1999}. According to \cite{Zhang1999}, the smooth curves $v=v^A(\rho)$ and $v=v^D(\rho)$ can intersect at multiple points.
For simplicity, we assume that these curves intersect at only one point $\rho=\rho_c$ and their relative positions are as in Figure \ref{fig:vu}.
The region $\{\rho<\rho_c\}$ is called {\it free zone}, while the other side $\{\rho>\rho_c\}$ is called
{\it congested zone}.

From the above description, we see that there is a scanning curve
through each point on the deceleration curve  $v=v^D(\rho)$. Similarly, there is also a scanning curve through each point on the acceleration curve $v=v^A(\rho)$. We parameterize these scanning curves with a parameter $h$. On each scanning curve the parameter $h$ is a constant, and then a scanning curve can be written as $v=v^S(\rho, h)$. We call $h$ the {\em hysteresis parameter}; it represents the driver's mood as
she drives according to the rule $v=v^S(\rho, h)$. We denote the range of $h$ as $[0, H]$. We note that as long as  scanning curves can be observed, their identifying parameter $h$ can be considered as  observable.

\begin{example}\label{ex_param} An example of parametrization of scanning curves is to assign $h$ to be
the $\rho$-coordinate of the intersection point between the curves $v=v^S(\rho,h)$ and $v=v^D(\rho)$.
In this case the range of $h$ is $[0, \rho_{\max}]$.
\end{example}

To describe the state of the driver at the point $(x, t)$ we use the variables $(\rho, h)(x, t)$. The hysteresis parameter $h(x, t)$ records the hysteresis mode the driver is in; hence, it should follow the driver and be updated as the driver's vehicle moves along the traffic. In other words, $h(x, t)$ records the scanning curve to take should the vehicle at $(x, t)$ switch into scanning mode.

We denote  by $\rho^D(h)$ (or $\rho^A(h)$) the density at the intersection of the curves $v=v^S(\rho, h)$ and $v=v^D(\rho)$ (or $v=v^A(\rho)$), see Figure \ref{fig:vu}.
The inverse functions of $\rho^D$ and $\rho^A$ are denoted as
$h=h^D(\rho)$ and $h=h^A(\rho)$, respectively.
They are the $h$ values of scanning curves intersecting $v=v^D(\rho)$ and $v=v^A(\rho)$
respectively.
For the simple parametrization in Example \ref{ex_param}, we have
$\rho^D(h) = h$, $h^D(\rho) = \rho$.
When $\rho=\rho_c$, the curves $v=v^A(\rho)$ and $v^D(\rho)$ intersect and hence
$h^D(\rho_c)=h^A(\rho_c)=:h_c$. At last,
\begin{align*}
v^A(\rho) &= v^S\left(\rho, h^A(\rho)\right),&
v^A\left(\rho^A(h)\right) &=  v^S\left(\rho^A(h), h\right),
\\
v^D(\rho) &= v^S\left(\rho, h^D(\rho)\right),&
v^D\left(\rho^D(h)\right) &=  v^S\left(\rho^D(h), h\right).
\end{align*}

The speed should be higher when the vehicle density $\rho$ is lower.
This translates to
\begin{equation}\label{Assumption 1.1}
v^A_\rho(\rho)<0, \quad  v^D_\rho(\rho)<0, \quad  v^S_\rho(\rho, h)< 0.
\end{equation}
At the intersection of $v=v^A(\rho)$ and $v=v^S(\rho, h)$, where $\rho= \rho^A(h)$, a vehicle follows the rule $v=v^S(\rho, h)$ when it decelerates from the speed $v=v^A(\rho)=v^S\left(\rho^A(h), h\right)$. Thus, the other side of the same scanning curve should have lower $v$ and higher $\rho$. This means that
\begin{equation}\label{Assumption 1.2}
\rho^A(h)\le\rho^D(h), \quad \text{ where $=$ holds if and only if $h=h_c$.}
\end{equation}
Notice that \eqref{Assumption 1.2} holds in both zones. Then, the physically relevant states lie in
\begin{equation}\label{physRegion}
\left\{ (\rho, h)\in [0,\rho_{\rm max}]\times[0,H]\  :  \rho^A(h) \le \rho \le \rho^D(h)\right\}.
\end{equation}
In the Example \ref{ex_param}, the higher $h$ is, the higher is the scanning curve $v^S(\cdot,h)$. We require that scanning curves do not intersect, and then we assume
\begin{equation}\label{Assumption 1.3}
v^S_h(\rho, h) >0.
\end{equation}
This assumption assures the invertibility of the functions $\rho^A$ and $\rho^D$
and allows us to uniquely define in $[0,\rho_{\rm max}]\times[0,H]$
\begin{equation}\label{generalV}
v = v(\rho, h) := \begin{cases}
v^A(\rho) &\text{ if } \rho \le \rho^A(h),
\\
v^S(\rho, h) &\text{ if } \rho^A(h)< \rho < \rho^D(h),
\\
v^D(\rho) &\text{ if } \rho\ge \rho^D(h).
\end{cases}
\end{equation}
We extended the velocity $v$ to the whole $[0,\rho_{\rm max}]\times[0,H]$ for technical reasons.

\smallskip
\noindent{\bf Assumption (I).} {\em Assume \eqref{Assumption 1.1},  \eqref{Assumption 1.2} and \eqref{Assumption 1.3}.
}

\smallskip

We define the boolean variables $A$, $D$ and $S$ as
\begin{gather*}
A= \left\{\rho = \rho^A(h),\   \rho_t<0\right\}, \qquad D= \left\{\rho = \rho^D(h), \   \rho_t>0\right\},
\\
S = \left\{ \rho^A(h) \le \rho\le  \rho^D(h), \   A={\it false} \hbox{ and } D={\it false}\right\}.
\end{gather*}
The boolean characteristic function is
\begin{equation*}
\chi(E) = \begin{cases}
1 &\text{ if } E \hbox{ is true, } \\
0 &\text{ else.}
\end{cases}
\end{equation*}
We observed that the hysteresis parameter $h$ should travel with the driver. In acceleration (or deceleration) mode, $h$ should be updated as $h^A(\rho)$ (or $h^D(\rho)$) along the vehicle path; in scanning mode, $h$ should not change. Thus, the equation for $h$ is
\begin{equation}\label{eq4h}
h_t + v h_x = \chi( A )\left(h^A(\rho)_t +vh^A(\rho)_x\right)
+ \chi( { D} )\left( h^D(\rho)_t + v h^D(\rho)_x\right).
\end{equation}
By combining  \eqref{e:LWR} and \eqref{eq4h} we obtain the system of equations modeling one-lane traffic flow with hysteresis:
\begin{equation}\label{systemInviscidh}
\left\{
\begin{array}{l}
\rho_t + \left(\rho v\right)_x = 0,
\\[2mm]
h_t + v h_x = \ds\chi( A )\left( h^A(\rho)_t +vh^A(\rho)_x\right)
+ \chi( { D} )\left( h^D(\rho)_t + v h^D(\rho) _x\right).
\end{array}
\right.
\end{equation}
\section{The model in Lagrangian coordinates}\label{s:Lagrange}
Tracing vehicle paths gives the advantage of seeing traffic from  drivers' point of view. Thus, converting \eqref{systemInviscidh} to Lagrangian coordinates offers some advantages, especially in light that the hysteresis parameter is updated along vehicle paths. If $\rho > 0$, Eulerian coordinates and Lagrangian coordinates are equivalent.

When using Lagrangian coordinates, the density $\rho$ is replaced by the spacing $u:=1/\rho$ between vehicles. We assume $\rho_{\rm max}=1$, without any loss of generality, so that $u\in[1,\infty)$. For sake of simplicity, with a slight abuse of notation we still use the same letter $h$ to denote the hysteresis parameter. Under this notation, we rewrite all previous functions $v^A$, $v^D$, $v^S$ into the new variables; we get $v^A(u)$, $v^D(u)$, $v^S(u, h)$, again with a slight abuse of notation. The functions corresponding to $\rho^A(h)$, $\rho^D(h)$, $h^A(\rho)$, $h^D(\rho)$ are transformed to $u^A(h)$, $u^D(h)$, $h^A(u)$, $h^D(u)$. If $h$ it is chosen as the $u$-coordinate of the intersection of the curves $v=v^D(u)$ and $v^S(u,h)$, then $h\in[1,\infty)$; for simplicity we keep this interval as the general range of $h$. We refer to Figure \ref{f:uv}.


\begin{figure}[htbp]
\begin{tikzpicture}[>=stealth]

\draw[->] (0,0) -- (12,0) node[below=0.2cm, left] {$u$} coordinate (x axis);
\draw[->] (0,0) -- (0,8) node[left] {$v$} coordinate (y axis);

\draw (0,7.5) node[left] {$\bar v$} --  (11,7.5);

\draw[thick] (1,0) node[below]{$1$} node[above=0.5cm, left]{$v^D$} .. controls (1.5,3) and (3,5.5) .. (11,6.5) node[right]{$v^D$};

\draw[thick] (11,7.3) node[right]{$v^A$} .. controls (5,7)  and (4,7.1) .. (3.5,0) node[below=0.3cm, left=-0.2cm]{$u_0^A$} node[above=0.5cm, left]{$v^A$};

\filldraw[fill=gray!20!white, draw=black, nearly transparent] (1,0) .. controls (1.5,3) and (3,5.5) ..  (11,6.5) -- (11,7.3) .. controls (5,7) and (4,7.1) .. (3.5,0) -- cycle;

\draw[dotted, thick] (4.22,0) node[below]{$u_c$} -- (4.22,4.65)
;
\draw[dotted,thick] (0,4.65) node[left]{$v_c$} -- (4.22,4.65);

\draw[thick] (1.21,1) .. controls (2,1.6) and (2.5,1.9) .. node[sloped,above=0.3cm, right=-0.5cm] {$v^S(\cdot,h_1)$} node[below=0.5cm]{$\mathcal{C}$}
(3.68,2);
\draw[->] (1.21,-0.6) -- (1.21,-0.1);
\draw[dotted, thick] (1.21,0) node[below=0.5cm]{$u^D(h)$} -- (1.21,1);
\draw[->] (3.68,-0.6) -- (3.68,-0.1);
\draw[dotted, thick] (3.68,0) node[below=0.5cm]{$u^A(h)$} -- (3.68,2);

\draw[thick] (6,5.44) .. controls (6.6,6.4)  .. node[sloped,left=0.9cm, above=0cm] {$v^S(\cdot,h_2)$} (7.5,7.06) node[right=0.8cm, below=0.2cm]{$\mathcal{F}$};
\draw[->] (6.0,-0.6) -- (6.0,-0.1);
\draw[dotted, thick] (6,0) node[below=0.5cm]{$u^D(h)$} -- (6,5.44);
\draw[->] (7.5,-0.6) -- (7.5,-0.1);
\draw[dotted, thick] (7.5,0) node[below=0.5cm]{$u^A(h)$} -- (7.5,7.06);

\end{tikzpicture}
\caption{{The functions $v^A(u)$ and $v^D(u)$, two scanning curves with $h_1<h_c<h_3$, the congested zone $\mathcal{C}$} and the free zone $\mathcal{F}$.}
\label{f:uv}
\end{figure}
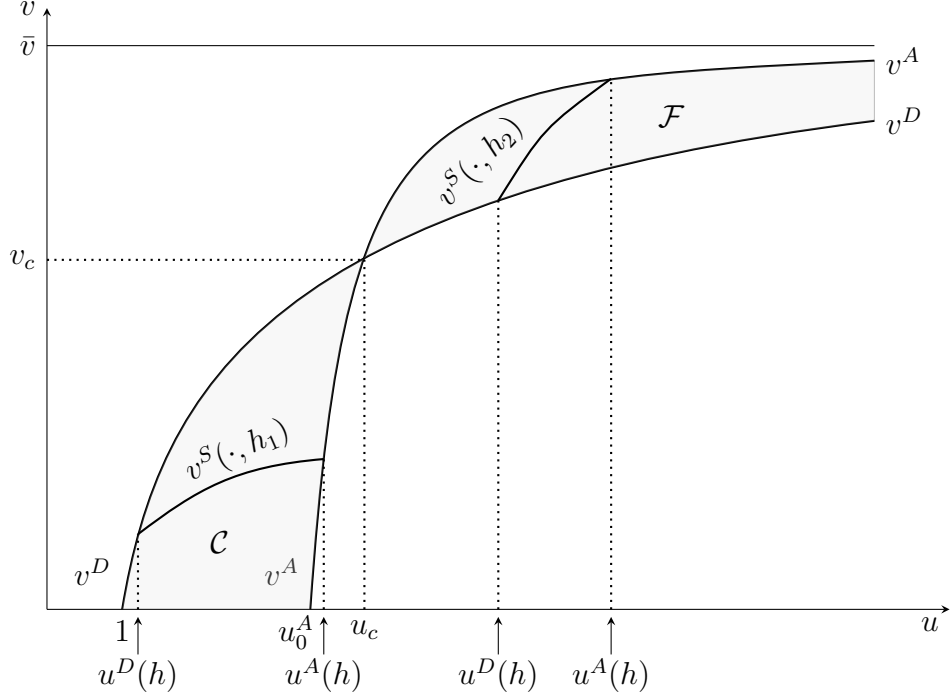

In Lagrangian coordinates Assumption (I) becomes

\smallskip

{\bf Assumption (I${}^\prime$)} The following conditions hold:
\begin{equation}\label{Assumption 1.1'}
\begin{array}{l}
v^A_u>0,\ v^D_u>0,\ v^S_u(u, h)>0,\  v^S_h(u, h)>0,
\\[2mm]
u^A(h) \ge u^D(h), \text{ with ``$=$" if and only if $h=h_c$.}
\end{array}
\end{equation}

\smallskip

In Lagrangian coordinates, the
system \eqref{systemInviscidh} becomes
\begin{equation}\label{system2h}
\left\{
\begin{array}{l}
u_t - v(u, h)_x = 0,
\\[2mm]
h_t = \chi(A) h^A(u)_t
+ \chi(D) h^D(u)_t,
\end{array}
\right.
\end{equation}
where
\begin{equation}\label{generalV(u)}
v = v(u, h) := \begin{cases}
v^D(u) &\text{ if } u \le u^D(h),
\\
v^S(u, h) &\text{ if } u^D(h)< u < u^A(h),
\\
v^A(u) &\text{ if } u\ge u^A(h),
\end{cases}
\end{equation}
and
\begin{gather*}
 A=\left\{u = u^A(h), u_t>0\right\},\qquad
 D= \left\{u = u^D(h)\ u_t<0\right\},
\\
S = \left\{ u^A(h) \ge u\ge  u^D(h),\    A={\it false} \hbox{ and } D={\it false}\right\}.
\end{gather*}
The region corresponding to \eqref{physRegion} is
\begin{equation}\label{physRegion2}
\Omega :=\left\{ (u, h)\in [1,\infty)\times[1,\infty) \  :\
u^D(h)\le u\le u^A(h) \right\}.
\end{equation}
Since the velocity $v$ takes on different forms, there could be possibilities for self-contradiction. For instance, at the intersection of scanning and acceleration mode, if the sign of $u_t$ calculated according to scanning mode and acceleration mode are different, then self-contradiction occurs.
To avoid such contradictions, the signs of the functions $v_u^A$, $v^D_u$, $v^S_u$ must be the same when the corresponding curves
$v^A(u)$, $v^D(u)$, $v^S(u, h)$ intersect,
i.e. these {\em compatibility conditions} must be required:
\begin{align*}
\sign v_u^A(u) = \sign v_u^S(u, h) \quad & \text{ when }h = h^A(u) \ne h_c,
\\
\sign v_u^D(u) = \sign v_u^S(u,h) \quad&\text{ when }h = h^D(u)\ne h_c.
\\
\sign v_u^D(u) = \sign v_u^A(u) \quad&\text{ when }h = h_c.
\end{align*}
Indeed, these conditions are implied by Assumption (I${}^\prime$).


\section{Basic waves}

In this section we find the shock and rarefaction waves needed to construct solutions for system \eqref{system2h}. A shock wave of \eqref{system2h} is a solution of the form
\begin{equation}\label{shock}
(u, h)(x, t)=
\left\{
\begin{array}{ll}
(u_-, h_-) & \hbox{ if } x-st < 0,
\\[2mm]
(u_+, h_+) & \hbox{ if } x-st > 0,
\end{array}
\right.
\end{equation}
where $(u_\pm, h_\pm)$ are constants and $s$ is the shock speed \cite{Dafermos}.
Typically, a shock wave is required to have a viscous profile derived from the corresponding viscous system
\begin{equation}\label{system2hv}
\left\{
\begin{array}{l}
u_t - v(u, h)_x = \e u_{xx},
\\[2mm]
h_t = \chi(A) h^A(u)_t
+ \chi(D) h^D(u)_t,
\end{array}
\right.
\end{equation}
and solutions of \eqref{system2hv} are expected to converge to solutions of \eqref{system2h} for $\e\to 0+$. Note that the second equation in \eqref{system2hv} has no viscosity because physically meaningless. The equations for the viscous profile of a shock wave connecting two states $(u_-,h_-)$ and $(u_+,h_+)$ with $u_-\ne u_+$ are, see \cite[(2.21)]{Corli-FanPM},
\begin{equation}\label{e:sp}
\left\{
\begin{array}{l}
u' = -s(u-u_\pm)-(v-v_\pm),
\\[2mm]
s h'=s \chi(A)\left(h^A(u)\right)' + s \chi(D)\left(h^D(u)\right)',
\\[3mm]
(u,h)(\pm\infty)=(u_\pm,h_\pm), \  (u', h')(\pm\infty)=(0,  0),
\end{array}
\right.
\end{equation}
where $' = \frac{\drm}{\drm\xi}$. The jump condition dictates
\begin{equation}\label{e:sigma}
s = - \frac{v(u_+, h_+) - v(u_-, h_-)}{u_+-u_-}.
\end{equation}
 With respect to the variable $\xi$, the boolean variables $A$ and $D$ become
\begin{equation}\label{e:twID}
A =  \left\{s u'<0 \text{ and } u=u^A(h)\right\},\qquad  D=\left\{s u'>0 \text{ and } u=u^D(h)\right\}.
\end{equation}

In the rest of this paper, we make the following assumption:

{\bf Assumption (H):} Assume Assumption (I${}^\prime$) and
\[
v^A_{uu}(u) < 0, \quad v^D_{uu}(u) < 0 \quad \hbox{ and }\quad v^S_{uu}(u,h)<0 \quad \hbox{ for } (u,h)\in\Omega.
\]
The second part of assumption (H) states the {\em concavity} of the functions $v^A$ and $v^D$, as in \cite{Newell1965, Yeo-Skabardonis}; the same requirement for $v^S$ is made by analogy. This assumption is not clearly justified by experimental data; it is made here for sake of simplicity and according to the modeling in the existing literature. However, we stress that what follows {\em does not essentially depend} on this assumption: the possible concavity of one curve and convexity of the other, or even the existence of inflection points only adds some technical difficulties but does not change the overall framework.

Notice that concavity is invariant when passing from Eulerian coordinates to Lagrangian coordinates. For brevity, we avoid to recall in the statements below that (H) is {\em always} assumed.

The following result is proved analogously to \cite[Lemma 4.1]{Corli-FanPM}.

\begin{lem}\label{lem:monotone} Assume $u_-\ne u_+$ and suppose that \eqref{e:sp} has a continuous solution $(u, h)$ with $s \neq 0$. Then
\begin{enumerate}[{(i)}]

\item $u'$ is continuous and never vanishes, hence $u$ is monotone;

\item $h$ is Lipschitz-continuous and monotone with the same type of monotonicity of $u$.
\end{enumerate}
\end{lem}

For future reference, we also write the equations that must be satisfied by a rarefaction wave \cite{Dafermos}. Denoting $\xi=x/t$ and with $u(\xi)$, $h(\xi)$ the components of the solution along the rarefaction, the equations are
\begin{equation}\label{e:rar}
\left\{
\begin{array}{l}
-\xi u' - \left(v(u,h)\right)'=0,
\\
-\xi h' = -\xi \chi(A)\left(h^A(u)\right)' - \xi \chi(D)\left(h^D(u)\right)',
\\
(u, h)(\pm\infty) = (u_\pm, h_\pm).
\end{array}
\right.
\end{equation}

\subsection{Acceleration and deceleration waves}
The next lemma shows that the only waves connecting two points $(u_\pm, v_\pm)$ on the acceleration curve are rarefaction waves; they are called {\em acceleration rarefaction waves}.
\begin{lem}\label{l:rar}
Assume $u_\pm = u^A(h_\pm)$. If $u_- < u_+$, then there is a rarefaction curve connecting $(u_-,h_-)$ with $(u_+, h_+)$. If $u_- > u_+$, then no connection is possible.
\end{lem}
\begin{proof}
Assume $u_-<u_+$. Let $h=h^A(u)$. Since $v=v^A(u)$, then
\eqref{e:rar}$_1$ becomes
\begin{equation}\label{scalarRW2}
-\xi =v^A_u(u).
\end{equation}
Because $v^A_{uu}<0$, the function $v^A_u$ can be inverted.
We define
\begin{equation}\label{RW}
u(\xi) =
\left\{
\begin{array}{ll}
u_- & \text{ if } \xi\le  -v^A_u(u_-),
\\
(v^A_u)^{-1}(-\xi)& \text{ if } -v^A_u(u_-) <\xi < -v^A_u(u_+),
\\
u_+ &  \text{ if } \xi\ge  -v^A_u(u_+),
\end{array}
\right.
\end{equation}
and $h(\xi ) := h^A\left(u(\xi)\right)$.
Then $(u, h)(\xi) $ satisfies
$\eqref{e:rar}_1$ and $\eqref{e:rar}_3$.
To show that it satisfies \eqref{e:rar}$_2$,
we notice that from \eqref{scalarRW2} it follows that $\xi<0$ and
$-1 = v^A_{uu} u'$; this implies $u'>0$, which infers $u_t = -\xi u'/t>0$ and hence $A={\it true}$.
Thus, equation $\eqref{e:rar}_2$ is also satisfied since $h=h^A(u)$.

In the case $u_->u_+$ there is no shock wave connecting $(u_-,v_-)$ with $(u_+,v_+)$ because the equation $\eqref{e:sp}_2$ for $h$ cannot be satisfied.
\end{proof}

Analogously, the only waves connecting $(u_\pm, v_\pm)$ on the deceleration curve are shock waves; they are called {\em deceleration shock waves}. This is the content of the following lemma, whose proof is analogous to that of Lemma \ref{l:rar}.

\begin{lem}\label{l:rar2}
Assume $u_\pm = u^D(h_\pm)$. If $u_- > u_+$, then there is a shock curve connecting $(u_-,h_-)$ with $(u_+,h_+)$. If $u_- < u_+$  then no connection is possible.
\end{lem}


\subsection{Scanning waves} If $(u_\pm,h_\pm)$ are on the same scanning curve then $h_-=h_+$. Equation $\eqref{system2h}_1$ becomes $u_t - v(u,h_-)_x=0$ while $\eqref{system2h}_2$ is trivially satisfied. Then system \eqref{system2h} reduces to a scalar conservation law.
Thus, any two points $(u_\pm, h_-=h_+)$ on a scanning curve
can be connected either by a shock, if $u_->u_+$, or by a rarefaction wave, if $u_-<u_+$.
Here, we used the condition $v^S_{uu}<0$ in Assumption (H). We collectively call these waves as {\em scanning waves}.


\subsection{Stationary shocks} These shocks have zero speed; the equations for the profiles are given by \eqref{e:sp} with $s=0$ and $v_-:=v(u_-,h_-)=v(u_+,h_+)=:v_+$ by \eqref{e:sigma}.

\begin{thm}
Consider two states $\left(u_\pm,h_\pm\right)\in\Omega$ such that $v_-=v_+$. Then system \eqref{e:sp} has infinitely many solutions with $s=0$ connecting $\left(u_-,h_-\right)$ to $\left(u_+,h_+\right)$.
\end{thm}

\begin{proof}
When $s=0$, the system $\eqref{e:sp}$ is reduced to
\begin{equation}\label{e:tw_st}
\left\{
\begin{array}{l}
u' = -v(u,h)+v_-,
\\[2mm]
(u,h)(\pm\infty) = \left(u_\pm,h_\pm\right).
\end{array}
\right.
\end{equation}
If $u_->u_+$, we select $h=h(u)$ in such a way that, see Figure \ref{f:stat-shock},
\begin{equation}\label{e:vvr}
h_\pm=h\left(u_\pm\right)\quad \hbox{ and }\quad v\left(u,h(u)\right)>v_-\quad \hbox{ for }u\in(u_+,u_-).
\end{equation}
%
%
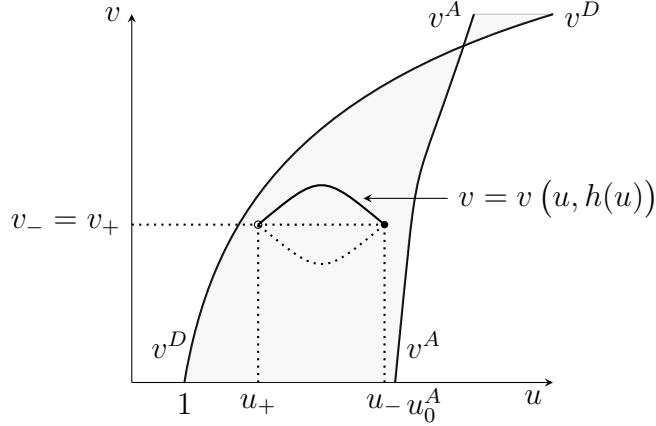
\begin{figure}[htbp]
\begin{tikzpicture}[>=stealth, scale=0.7]


\draw[->] (0,0) -- (8,0) node[below=0.2cm, left] {$u$} coordinate (x axis);
\draw[->] (0,0) -- (0,7) node[left] {$v$} coordinate (y axis);

\draw[thick] (1,0) node[below]{$1$} node[above=0.5cm, left=-0.2cm]{$v^D$} .. controls (1.5,3) and (3,5.5) .. (8,7) node[right]{$v^D$};

\draw[thick] (6.5,7) node[left]{$v^A$} .. controls (5.2,3)  and (5.5,5) .. (5,0) node[below=0.3cm, right]{$u_0^A$} node[above=0.5cm, right]{$v^A$};

\filldraw[fill=gray!20!white, draw=black, nearly transparent] (1,0) .. controls (1.5,3) and (3,5.5) ..  (8,7) -- (6.5,7) .. controls (5.2,3) and (5.5,5) .. (5,0) -- cycle;

\draw[thick] (4.8,3)  .. controls (3.6,4) .. (2.4,3);
\draw[thick, dotted] (4.8,3)  .. controls (3.6,2) .. (2.4,3);
\draw[thick, dotted] (4.8,0) node[below]{$u_-$} -- (4.8,3);
\draw[thick, dotted] (2.4,0) node[below]{$u_+$} -- (2.4,3);
\draw[thick, dotted] (0,3) node[left]{$v_-=v_+$} -- (4.8,3);

\fill (4.8,3) circle (2pt);
\draw (2.4,3) circle (2pt);

\draw[->] (6,3.5) node[right]{$v=v\left(u,h(u)\right)$} -- (4.4,3.5);

\end{tikzpicture}
\caption{{States $(u_-,v_-)$ and $(u_+,v_+=v_-)$ can be connected by a stationary shock.}}
\label{f:stat-shock}
\end{figure}

Then the unstable trajectory issued from the equilibrium point $(u_-,h_-)$ and entering into the region $\{u<u_-\}$ will continue to decrease as $\xi$ increases, until it enters another equilibrium point $(u_+,h_+)$ at $\xi=\infty$. Then $\left(u(\xi), h(u(\xi))\right)$ is indeed a solution to  \eqref{e:vvr}.

If $u_-<u_+$, the proof is analogous but choosing $h=h(u)$ such that $v\left(u,h(u)\right)<v_-$ in \eqref{e:vvr}.
\end{proof}


\subsection{Scanning to acceleration or deceleration shocks}
In this section we focus on shock waves connecting a state on a scanning curve with a state located either on the acceleration curve or on the deceleration curve.  We refer to Figure \ref{f:acc-dec-shock}.

\begin{figure}[htbp]

\begin{tabular}{cc}

\begin{tikzpicture}[>=stealth, scale=0.7]

\draw[->] (0,0) -- node[midway, below=0.6cm]{$(a)$} (8,0) node[below=0.2cm, left] {$u$} coordinate (x axis);
\draw[->] (0,0) -- (0,7) node[left] {$v$} coordinate (y axis);

\draw[thick] (1,0) node[below=0.3cm,left=0.0cm]{$1$} node[above=0.5cm, left=-0.2cm]{$v^D$} .. controls (1.5,3) and (3,5.5) .. (8,7) node[right]{$v^D$};

\draw[thick] (6.5,7) node[left]{$v^A$} .. controls (5.2,3)  and (5.5,5) .. (5,0) node[above=4cm, right=0.7cm]{$v^A$};

\filldraw[fill=gray!20!white, draw=black, nearly transparent] (1,0) .. controls (1.5,3) and (3,5.5) ..  (8,7) -- (6.5,7) .. controls (5.2,3) and (5.5,5) .. (5,0) -- cycle;

\draw[thick] (1.2,1) .. node[midway, sloped, below]{$h=h_-$} controls (3,2)  and (4,1.8) .. (5.2,2);
\draw[thick,dotted] (1.2,0) node[below=0.3cm,right=-0.3cm]{$u^D(h_-)$} -- (1.2,1);
\draw[thick,dotted] (5.2,0) node[below]{$u^A(h_-)$} -- (5.2,2);

\draw[thick] (2.1,3.1) .. node[near end, sloped, above]{$h=h_+$} controls (3,3.5)  and (3,3.7) .. (5.48,4);

\path [draw=black,thick,postaction={on each segment={mid arrow=black,thick}}]
(3,1.72) -- (5.47,3.99);
\draw[thick,dotted] (5.48,0) node[above=0.2cm,right=0.0cm]{$u_+$} -- (5.48,4);
\draw[thick,dotted] (3,0) node[below=0.3cm,right=-0.1cm]{$u_-$} -- (3,1.72);

\fill (3,1.72) circle (2pt);
\draw (5.48,4) circle (2pt);
\end{tikzpicture}
\begin{tikzpicture}[>=stealth, scale=0.7]
\draw[->] (0,0)  -- node[midway, below=0.6cm]{$(b)$} (8,0) node[below=0.2cm, left=-0.2cm] {$u$} coordinate (x axis);
\draw[->] (0,0) -- (0,7) node[left] {$v$} coordinate (y axis);

\draw[thick] (1,0) node[below=0.3cm,left=0.0cm]{$1$} node[above=0.5cm, left]{$v^D$} .. controls (1.5,3) and (3,5.5) .. (8,7) node[right]{$v^D$};

\draw[thick] (6.5,7) node[left]{$v^A$} .. controls (5.2,3)  and (5.5,5) .. (5,0) node[below]{$u_0^A$} node[above=0.5cm, right]{$v^A$};

\filldraw[fill=gray!20!white, draw=black, nearly transparent] (1,0) .. controls (1.5,3) and (3,5.5) ..  (8,7) -- (6.5,7) .. controls (5.2,3) and (5.5,5) .. (5,0) -- cycle;

\draw[thick] (1.2,1) .. node[midway, sloped, below]{$h=h_+$} controls (3,2)  and (4,1.8) .. (5.2,2);
\draw[thick,dotted] (1.2,0) node[below=0.3cm,right=-0.4cm]{$u_+$} -- (1.2,1);

\draw[thick] (2.1,3.1) .. node[near end, sloped, above]{$h=h_-$} controls (3,3.5)  and (3,3.7) .. (5.48,4);
\draw[thick,dotted] (2.1,0) node[below=0.3cm,right=-0.5cm]{$u^D(h_-)$}-- (2.1,3.1);

\path [draw=black,thick,postaction={on each segment={mid arrow=black,thick}}]
(4.05,3.8) -- (1.23,1.03);
\draw[thick,dotted] (4,0) node[below]{$u_-$} -- (4,3.8);

\fill (4,3.8) circle (2pt);
\draw (1.2,1) circle (2pt);

\end{tikzpicture}

\end{tabular}
\caption{{$(a)$: An acceleration shock, and $(b)$: a deceleration shock connecting states $(v_-, h_-)$ and $(v_+, h_+)$.}}
\label{f:acc-dec-shock}
\end{figure}
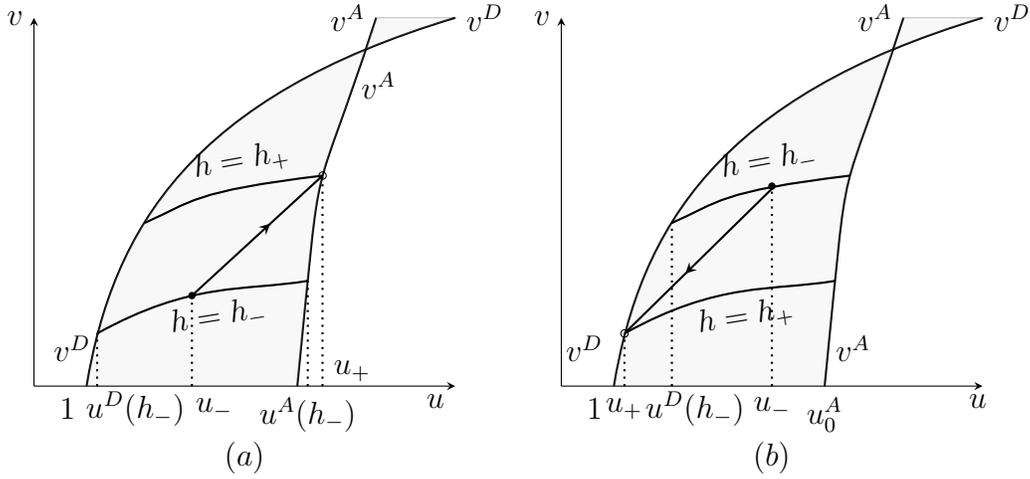

\begin{thm}\label{t:scz1}
Assume $(u_-,h_-)\in\Omega$, $u_-<u_+=u^A(h^+)$, $h_-<h_+$ and $u^D(h_-)\le u_- < u^A(h_-)$. If in the interval $(u_-,u_+)$ the chord condition
\begin{equation}\label{chord_S2A}
\frac{v_+-v_-}{u_+-u_-} > \frac{v-v_-}{u-u_-}
\end{equation}
is satisfied for $v=v(u,h_-)$ and $v=v^A(u)$, then there is a traveling wave connecting $(u_-,h_-)$ with $(u_+,h_+)$.
\end{thm}
\begin{proof}
We refer to Figure \ref{f:acc-dec-shock}{\em (a)}. We define
\begin{equation}\label{e:hscz1}
h(u) =
\left\{
\begin{array}{ll}
h_- &\hbox{ if }u\in\left[u_-,u^A(h_-)\right],
\\[2mm]
h^A(u) &\hbox{ if }u\in\left[u^A(h_-),u_+\right].
\end{array}
\right.
\end{equation}
Then
\eqref{e:sp}$_1$ has an increasing solution with $u(\pm\infty)=u_\pm$
thanks to the chord condition \eqref{chord_S2A}.
This gives $(u, h) = \left(u(\xi), h(u(\xi))\right)$ which satisfies \eqref{e:sp}$_1$.
To see that it also satisfies
$\eqref{e:sp}_2$,
we notice that   when $\xi$ increases from $-\infty$,
the pair $(u, h)(\xi)$ moves along the scanning curve $v=v^S(u, h_-)$ until some point $\xi=\xi_0$ where $u(\xi_0)=u^A(h_-)$.
Equation $\eqref{e:sp}_2$ is trivially satisfied for $\xi\in(-\infty, \xi_0)$.
Over the interval $\xi\in (\xi_0, \infty)$,
we have $(u, h)(\xi) = \left(u, h^A(u)\right)(\xi)$ and $u'>0$; hence $\chi(A)=1$ holds, making
the equation $\eqref{e:sp}_2$ satisfied.
At the point $\xi=\xi_0$ where the two parts join, the pair $(u, h)(\xi)$ is continuous. Thus,
$\eqref{e:sp}_2$ is satisfied for all $\xi\in(-\infty, \infty)$.
\end{proof}

\begin{remark}\label{S2D and S2A}
The above theorem establishes the existence of acceleration shocks. Notice that LWR model does not allow acceleration shocks. 
Acceleration shocks are not mentioned in existing literature as far as we know. However, many observations of real traffic, see e.g. Fig 3 of \cite{Kaufmann_et_al2017}, show that the acceleration front of a stop-and-go patterns can be as sharp as the deceleration shock serving as the other front, and the acceleration front is not visually expanding like a rarefaction wave. Thus, we think this sharp acceleration front is an acceleration shock. 
 
The conditions of \thmref{t:scz1} cannot hold in free zone $u> u_c$ because of the concavity of $v=v^A(u)$.
\end{remark}

\begin{thm}\label{t:scz2}
Assume $\left(u_-,h_-\right)\in \Omega$,  $h_->h_+\ge1$ and $u_+ = u^D(h_+)$.
If for every $u\in(u_+,u_-)$ the chord condition
\begin{equation}\label{chord_S2D}
\frac{v_+-v_-}{u_+-u_-} > \frac{v(u, h_-)-v_-}{u-u_-}
\end{equation}
holds, then there is a shock wave connecting $(u_-,h_-)$ with $(u_+,h_+)$.
\end{thm}
\begin{proof}
We refer to Figure \ref{f:acc-dec-shock}{\em (b)}. Because of Assumption (H), the conditions $h_+<h_-$ and $u_+ =u^D(h_+)$ imply $u_+ < u^D(h_-)\le u_-$.
The system \eqref{e:sp} becomes
\begin{equation}\label{e:spp}
\left\{
\begin{array}{l}
u' = (u-u_\pm)\ds\left(-s-\frac{v(u, h)-v_\pm}{u-u_\pm}\right),
\\
h'=\chi(A)\left(h^A(u)\right)' + \chi(D)\left(h^D(u)\right)'.
\end{array}
\right.
\end{equation}
We define
\begin{equation}\label{e:htr}
h(u) =
\left\{
\begin{array}{ll}
h_- &\hbox{ if }u\in\left[u^D(h_-),u_-\right],
\\[2mm]
h^D(u) &\hbox{ if }u\in\left[u_+,u^D(h_-)\right],
\end{array}
\right.
\end{equation}
and plug it into \eqref{e:spp}. Then \eqref{e:spp}$_1$ is an equation with the only unknown $u$.
It has a trajectory leaving $u_-$ and entering into the $\{u<u_-\}$ side, which keeps decreasing due to the  chord condition. This trajectory will enter the equilibrium point $u_+$.
The proof that $\left(u, h(u)\right)(\xi)$ also satisfies \eqref{e:spp}$_2$ is as in the proof of
\thmref{t:scz1}.
\end{proof}

\begin{remark}
	The speed of stationary shock is $0$. All other waves listed in this section have negative speed $s<0$ when their corresponding sufficient conditions are satisfied.
\end{remark}
\section{Riemann solvers}\label{s:Riemann-solvers}
In this section we find solutions of system
\eqref{system2h} for Riemann initial data
\begin{equation}\label{e:RData}
(u, h)(x, 0)=
\left\{
\begin{array}{ll}
(u_-, h_-) &\hbox{ if } x<0,
\\[2mm]
(u_+, h_+) &\hbox{ if } x>0.
\end{array}
\right.
\end{equation}
We call such solutions Riemann solvers. We use for brevity the following notations: S denotes shock waves, R rarefaction waves, ST stationary waves, ScS scanning shock and ScR scanning rarefaction waves, AR acceleration rarefaction waves, DS deceleration shocks, ScDS scanning-to-deceleration and ScAS scanning-to-acceleration shocks. We denote $v_\pm=v(u_\pm,h_\pm)$ and describe solutions in the $(u,v)$-plane, instead of using the $(u,h)$-plane of the states, because this is more intuitive, and by assumption (H) the two approaches are clearly equivalent.


\subsection{Case 1. $u_-\ge u_c$}
In this case the pair $(u_-, h_-)\in\Omega$ is in free zone and $h_-$ is the $h$-value of the scanning curve passing through $(u_-,v_-)$ in the $(u,v)$-plane. The Riemann solver is of the following forms; we refer to Figure \ref{f:fromfreezone}.

\begin{enumerate}[{\em (i)}]
\item If $v_+ \ge v_-$, then
\[
(u_-,v_-) \xrightarrow{\text{ScR}}
\left(u^A(h_-), v^A(u^A(h_-))\right) \xrightarrow{\text{AR}}
\left(u_m, v_+= v^A(u_m)\right) \xrightarrow{\text{ST}}
(u_+,v_+).
\]
Some waves may be missing in the above scheme; for instance, only the first and third wave are needed if $v\le v^S\left(u^A(h_-),h_-\right)$.

\item If $v_+< v_-$, then

\begin{enumerate}[{\em (a)}]

\item either
\[
(u_-,v_-) \xrightarrow{\text{ScS}}
\left(u^D(h_-), v^D(u^D(h_-))\right) \xrightarrow{\text{DS}}
\left(u_m, v_+=v^D(u_m)\right) \xrightarrow{\text{ST}}
(u_+,v_+),
\]
if there exists $\left(u_m, v_+=v^D(u_m)\right)$ such that the speed of the deceleration shock DS is less then that of the scanning wave. Since all velocities are negative, the lesser the velocity the steeper is the chord joining the states.

\item or, if the previous condition is not met, then
\[
(u_-,v_-) \xrightarrow{\text{ScDS}} \left(u_m, v_+=v^D(u_m)\right) \xrightarrow{\text{ST}}
(u_+,v_+).
\]
\end{enumerate}
\end{enumerate}


\begin{figure}[htbp]
\begin{tikzpicture}[>=stealth, scale=0.9]


\draw[->] (0,0) -- (12,0) node[below=0.2cm, left] {$u$} coordinate (x axis);
\draw[->] (0,0) -- (0,8) node[left] {$v$} coordinate (y axis);

\draw (0,7.5) node[left] {$\bar v$} --  (11,7.5);

\draw[thick] (1,0) node[below]{$1$} node[above=0.5cm, left]{$v^D$} .. controls (1.5,3) and (3,5.5) .. (11,6.5) node[right]{$v^D$};

\draw[thick] (11,7.3) node[right]{$v^A$} .. controls (5,7)  and (4,7.1) .. (3.5,0) node[below=0.3cm, left=-0.2cm]{$u_0^A$} node[above=0.5cm, left]{$v^A$};

\filldraw[fill=gray!20!white, draw=black, nearly transparent] (1,0) .. controls (1.5,3) and (3,5.5) ..  (11,6.5) -- (11,7.3) .. controls (5,7) and (4,7.1) .. (3.5,0) -- cycle;

\draw[dotted, thick] (4.22,0) node[below]{$u_c$} -- (4.22,4.65);
\draw[dotted,thick] (0,4.65) node[left]{$v_c$} -- (4.22,4.65);

\draw[thick] (5.33,5.2) .. controls (6.8,6.4) and (6.8,6.5) ..  (8,7.1);
\draw[dotted, thick] (5.33,0) node[below]{$u^D(h_-)$} -- (5.33,5.2);
\draw[dotted, thick] (8,0) node[below]{$u^A(h_-)$} -- (8,7.1);
\draw[dotted, thick] (5.33,5.2) -- (1.65,2.15); 
\draw[->] (10,5) node[below]{$v=v^S(u,h_-)$} -- (6.9,6.3);
\fill (6.6,6.25) circle (2pt);
\draw[thick,dotted] (6.6,0) node[below]{$u_-$}-- (6.6,6.25);
\draw[thick,dotted] (0,6.25) node[left]{$v_-$}-- (6.6,6.25);

\draw[thick,->] (6.95,6.5) -- (7.07,6.6); 
\path [draw=black,thick,postaction={on each segment={mid arrow=black,thick}}]
(7.25,6.69) -- (8.3,6.69);
\draw (8.3,6.69) circle (2pt);

\draw[thick,->] (8.5,7.16) -- (8.55,7.17); 
\path [draw=black,thick,postaction={on each segment={mid arrow=black,thick}}]
(8.8,7.18) -- (10.5,7.18);
\draw (10.5,7.18) circle (2pt);

\draw[thick,->] (6.25,5.95) -- (6.20,5.9); 
\path [draw=black,thick,postaction={on each segment={mid arrow=black,thick}}]
(5.95,5.7) -- (5.1,5.7);
\draw (5.1,5.7) circle (2pt);

\path [draw=black,thick,postaction={on each segment={mid arrow=black,thick}}]
(5.33,5.2) -- (2.6,3.5) 
(2.6,3.5) -- (3.6,3.5);
\draw (3.6,3.5) circle (2pt);

\path [draw=black,thick,postaction={on each segment={mid arrow=black,thick}}]
(6.6,6.25) -- (1.05,0.3) 
(1.05,0.3) -- (2.5,0.3); 
\draw (2.5,0.3) circle (2pt);

\end{tikzpicture}
\caption{{Riemann solutions in Case 1. The state $(u_-,v_-)$ is represented by a full circle; various states $(u_+,v_+)$ are represented by empty circles. The oblique dotted line, whose slope is $v_u^S\left(u^D(h_-),h_-\right)$, separates case {\em (ii)(a)}, which occurs above that line, from case {\em (ii)(b)}, which occurs below.}}
\label{f:fromfreezone}
\end{figure}
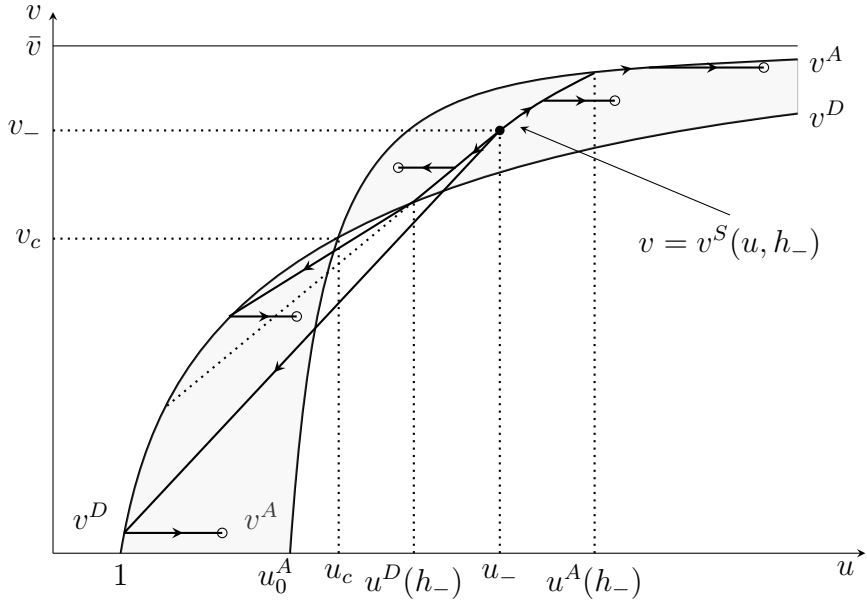

\subsection{Case 2. $u_- < u_c$} In this case $u_-$ is in congested zone. As in the previous case, solutions are classified depending on $v_+$. We refer to Figure \ref{f:fromcongestedzone}.


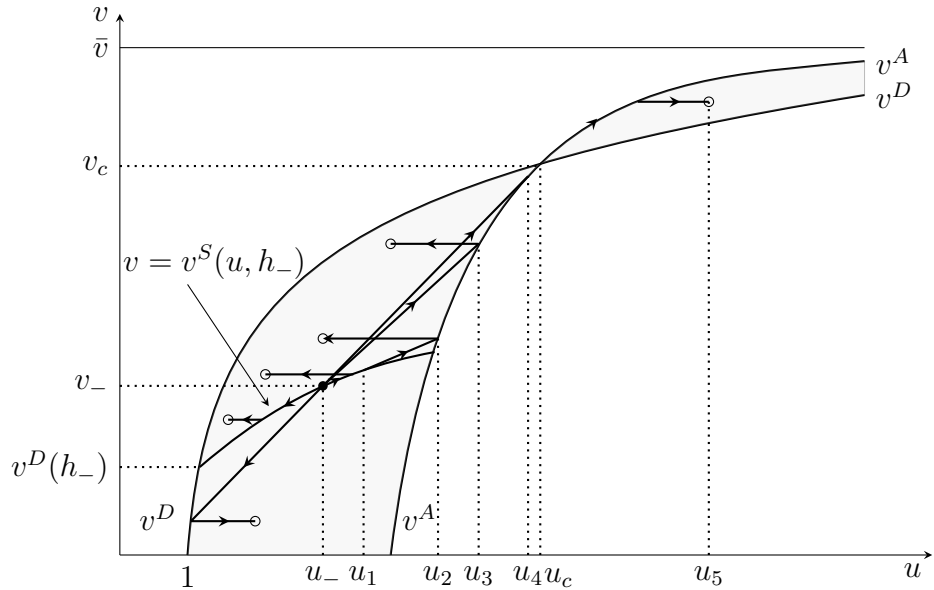
\begin{figure}[htbp]
\begin{tikzpicture}[>=stealth, scale=0.9]


\draw[->] (0,0) -- (12,0) node[below=0.2cm, left] {$u$} coordinate (x axis);
\draw[->] (0,0) -- (0,8) node[left] {$v$} coordinate (y axis);

\draw (0,7.5) node[left] {$\bar v$} --  (11,7.5);

\draw[thick] (1,0) node[below]{$1$} node[above=0.5cm, left]{$v^D$} .. controls (1.3,4) and (3,5.5) .. (11,6.8) node[right]{$v^D$};

\draw[thick] (11,7.3) node[right]{$v^A$} .. controls (8,7)  and (4.9,7.1) .. (4,0) node[above=0.5cm, right]{$v^A$};

\filldraw[fill=gray!20!white, draw=black, nearly transparent] (1,0) .. controls (1.3,4) and (3,5.5) ..  (11,6.8) -- (11,7.3) .. controls (8,7) and (4.9,7.1) .. (4,0) -- cycle;

\draw[dotted, thick] (6.21,0) node[below=0.3cm,right=-0.1cm]{$u_c$} -- (6.21,5.75);
\draw[dotted,thick] (0,5.75) node[left]{$v_c$} -- (6.21,5.75);

\fill (3,2.5) circle (2pt);
\draw[dotted, thick] (3,0) node[below]{$u_-$} -- (3,2.5);
\draw[dotted,thick] (0,2.5) node[left]{$v_-$} -- (3,2.5);

\draw[->] (1,4) node[above=0.3cm, right=-1cm]{$v=v^S(u,h_-)$} -- (2.2,2.2);
\draw[thick] (1.18,1.3) .. controls (2,2) and (3,2.7) ..  (4.65,3);
\draw[thick,dotted] (0,1.3) node[left]{$v^D(h_-)$} -- (1.18,1.3);

\draw[thick,->] (2.5,2.24) -- (2.42,2.19); 
\path [draw=black,thick,postaction={on each segment={mid arrow=black,thick}}]
(2.1,2) -- (1.6,2);
\draw (1.6,2) circle (2pt);

\path [draw=black,thick,postaction={on each segment={mid arrow=black,thick}}]
(3,2.5) -- (1.05,0.5)
(1.05,0.5) -- (2,0.5);
\draw (2,0.5) circle (2pt);

\draw[thick,->] (3.2,2.59) -- (3.25,2.62); 
\path [draw=black,thick,postaction={on each segment={mid arrow=black,thick}}]
(3.45,2.67) -- (2.15,2.67);
\draw (2.15,2.67) circle (2pt);

\path [draw=black,thick,postaction={on each segment={mid arrow=black,thick}}]
(3.6,2.73) -- (4.7,3.2);
\draw[thick,dotted] (3.6,0) node[below]{$u_1$} -- (3.6,2.73);
\draw[thick,dotted] (4.7,0) node[below]{$u_2$} -- (4.7,3.2);
\draw[thick, ->] (4.7,3.2) -- (3,3.2);
\draw (3,3.2) circle (2pt);

\path [draw=black,thick,postaction={on each segment={mid arrow=black,thick}}]
(3,2.5) -- (5.3,4.6)
(5.3,4.6) -- (4,4.6);
\draw[thick,dotted] (5.3,0) node[below]{$u_3$} -- (5.3,4.6);
\draw (4,4.6) circle (2pt);

\draw[thick] (3,2.5) -- (6.03,5.6);
\draw[thick,->] (5.15,4.7) -- (5.25,4.8); 
\draw[thick,dotted] (6.03,0) node[below]{$u_4$} -- (6.03,5.6);
\draw[thick,->] (7,6.4) -- (7.05,6.45); 
\path [draw=black,thick,postaction={on each segment={mid arrow=black,thick}}]
(7.65,6.7) -- (8.7,6.7) ;
\draw[thick,dotted] (8.7,0) node[below]{$u_5$} -- (8.7,6.7);
\draw (8.7,6.7) circle (2pt);

\end{tikzpicture}
\caption{{Riemann solutions in Case 2. Notations are as in Figure \ref{f:fromfreezone}.}}
\label{f:fromcongestedzone}
\end{figure}

\begin{enumerate}[{\em (i)}]
\item If $v_+\le v^D(h_-)$, then
\[
(u_-,v_-) \xrightarrow{\text{ScDS}} \left(u_m, v_+ = v^D( u_m)\right)
\xrightarrow{\text{ST}} (u_+, v_+).
\]

\item If $v^D(h_-) \le v_+ \le v_-$, then
\[
(u_-,v_-) \xrightarrow{\text{ScS}} \left(u_m, v_+=v^S(u_m, h_-)\right)
\xrightarrow{\text{ST}} (u_+, v_+).
\]

\item
If $v_+> v_- $, then the solution is one of the following, in order of increasing $v_+$:

\begin{enumerate}[{\em (a)}]

\item either
\[
(u_-,v_-) \xrightarrow{\text{ScR}} \left(u_m, v_+ = v^S(u_m, h_-)\right)
\xrightarrow{\text{ST}} (u_+, v_+);
\]

\item or
\[
(u_-,v_-) \xrightarrow{\text{ScR}} \left(u_1, v_1 = v^S( u_1, h_-)\right)
\xrightarrow{\text{ScAS}}  \left(u_2, v^A( u_2)=v_+\right)
\xrightarrow{\text{ST}} (u_+, v_+),
\]
where the the chord connecting $(u_1, v_1)$ and $(u_2, v_2)$ is tangent to
the curve $v=v^S(u, h_-)$ at $u=u_1$;

\item or
\[
(u_-,v_-)
\xrightarrow{\text{ScAS}}  \left(u_3, v^A( u_3)=v_+\right)
\xrightarrow{\text{ST}} (u_+, v_+);
\]

\item or, at last,
\[
(u_-,v_-)
\xrightarrow{\text{ScAS}}  \left(u_4, v^A( u_4)\right)
\xrightarrow{\text{AR}} \left(u_5, v_+= v^A(u_5)\right)
\xrightarrow{\text{ST}} (u_+, v_+),
\]
where the chord connecting $(u_-, v_-)$ and $\left(u_4, v^A( u_4)\right)$ is tangent to the curve
$v=v^A(u)$ at $u=u_4$.
\end{enumerate}
\end{enumerate}

\noindent This concludes the analysis of the possible Riemann solutions.

\smallskip

Stationary shocks, their existence and frequency of appearing in Riemann solvers listed above, seem strange.
They are not present in classical macroscopic models
such as the LWR, 
though they may occur in second-order models as Aw-Rascle-Zhang \cite{Aw-Rascle, Zhang2002}.
Here, we give an intuitive explanation for the existence of stationary shocks as the result of rational drivers' behavior.

Consider a car train traveling in the congested zone $u<u_c$ with $(u_\pm, v_\pm)\in\Omega$ and $v_-=v_+$ in the initial data \eqref{e:RData}. The states $(u_\pm, v_\pm)$ are connected by a stationary shock at $x=0\pm$, i.e. a shock traveling with the
cars. Note that, in Lagrangian coordinates, each car has a fixed $x$-coordinate.
For a rational driver in the platoon, she sees no reason to increase the speed because, if she does, then the front car traveling at $v_\pm$ will force her to brake soon with less safe and less comfortable spacing in front of her, wasting her effort and fuel. On the other hand, she has no incentive to drive slower either, unless she wants to increase the spacing in front of her which invites cars in other lanes to cut in.
Therefore, the stationary shock persists.

Why do stationary shocks appear so often in Riemann solvers listed above?
Are there other Riemann solvers for the same initial data?
Indeed, for many initial data \eqref{e:RData}, there are infinitely many Riemann solvers, and some of them do not use stationary shocks. However, we claim that other solutions are a waste of drivers' effort and fuel without much to gain  for both the driver initiating
those Riemann solvers and those drivers behind.
To see this, consider the initial condition \eqref{e:RData} with $(u_\pm,v_\pm)$  as in Figure \ref{f:SNGwaves}.
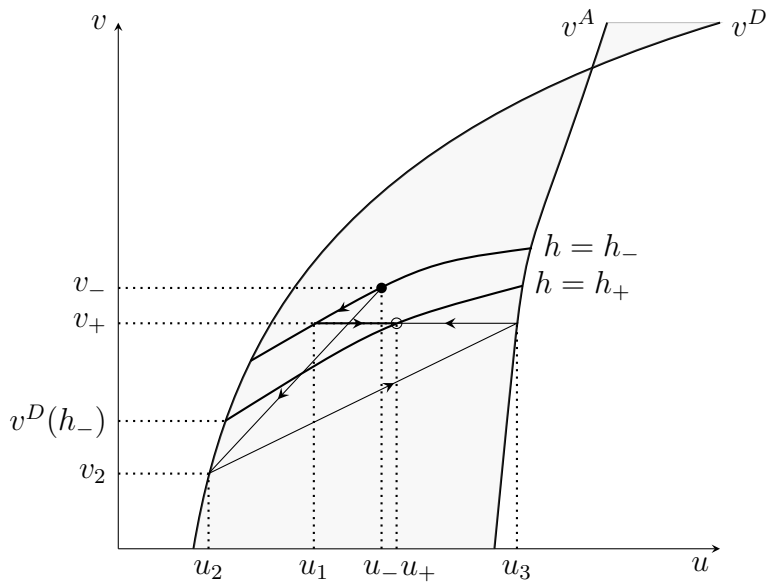
\begin{figure}[htbp]
\begin{tikzpicture}[>=stealth], scale=0.7]


\draw[->] (0,0) -- (8,0) node[below=0.2cm, left] {$u$} coordinate (x axis);
\draw[->] (0,0) -- (0,7) node[left] {$v$} coordinate (y axis);

\draw[thick] (1,0) .. controls (1.5,3) and (3,5.5) .. (8,7) node[right]{$v^D$};

\draw[thick] (6.5,7) node[left]{$v^A$} .. controls (5.2,3)  and (5.5,5) .. (5,0);

\filldraw[fill=gray!20!white, draw=black, nearly transparent] (1,0) .. controls (1.5,3) and (3,5.5) ..  (8,7) -- (6.5,7) .. controls (5.2,3) and (5.5,5) .. (5,0) -- cycle;

\draw[thick] (1.76,2.5)  .. controls (4,3.8) .. (5.5,4) node[right]{$h=h_-$};

\draw[thick] (1.42,1.7)  .. controls (3.5,3) .. (5.38,3.5) node[right]{$h=h_+$};
\draw[thick, dotted] (0,1.7) node[left]{$v^D(h_-)$}-- (1.42,1.7);

\fill (3.5,3.47) circle (2pt);
\draw[thick, dotted] (3.5,0) node[below]{$u_-$}-- (3.5,3.47);
\draw[thick, dotted] (0,3.47) node[left]{$v_-$}-- (3.5,3.47);
\draw[thick,->] (3,3.2) -- (2.92,3.15); 
\path [draw=black,thick,postaction={on each segment={mid arrow=black,thick}}]
 (2.6,3) -- (3.7,3);
\draw[thick, dotted] (2.6,0) node[below]{$u_1$}-- (2.6,3);
\draw (3.7,3) circle (2pt); 
\draw[thick, dotted] (3.7,0) node[below=0.3cm,right=-0.1cm]{$u_+$}-- (3.7,3);
\draw[thick, dotted] (0,3) node[left]{$v_+$}-- (3.7,3);

\path [draw=black,postaction={on each segment={mid arrow=black,thick}}]
(3.5,3.47) -- (1.2,1) -- (5.3,3) -- (3.7,3);
\draw[thick, dotted] (5.3,0) node[below]{$u_3$}-- (5.3,3);
\draw[thick, dotted] (1.2,0) node[below]{$u_2$} -- (1.2,1);
\draw[thick, dotted] (0,1) node[left]{$v_2$} -- (1.2,1);
\end{tikzpicture}
\caption{{Two different solutions for the same initial data. The solution proposed in Case 2{\em (ii)} is depicted with thick lines, the other one with thin lines.}}
\label{f:SNGwaves}
\end{figure}
Because $v_- > v_+$, the driver of the first car at the $\{x<0\}$ part will have to decelerate.
The Riemann solver listed in Case 2{\em (ii)} for this initial data is
\begin{equation}\label{RSolver1}
(u_-,v_-)
\xrightarrow{\text{ScS}}  \left(u_1, v_+\right)
\xrightarrow{\text{ST}} (u_+, v_+),
\end{equation}
which says that the driver first slows down to state $(u_1, v_+)$ on the scanning curve
$h=h_-$ to speed $v_+$, the same as the car in front of him.
Once the speed is slowed down to $v_+$, there is no incentive for him to slow down more, and hence he will keep the speed $v_+$ afterwards,
resulting in the stationary shock at $x=0$.
This rational driver's behavior is the reason for stationary shocks appearing in Riemann solvers.

However, as long as the scanning curves' slopes are relatively small,
there are also infinitely many other Riemann solvers of the form
\begin{equation}\label{RSolver2}
(u_-,v_-)
\xrightarrow{\text{ScDS}}  \left(u_2, v_2\right)
\xrightarrow{\text{ScR or ScAS}} (u_3, v_+)
\xrightarrow{\text{ST}} (u_+, v_+),
\end{equation}
as illustrated in Figure \ref{f:SNGwaves} in the case the second wave is a scanning-to-acceleration shock.
This solution form corresponds to the situation where the driver of the first car at $x=0-$
overdecelerates to speed $v_2<v_+$, possibly due to applying brake late or to other random reasons, forcing cars behind him also to decelerate to the velocity $v_2$ and with a denser spacing $u_2$. Because $v_2<v_+$, the spacing in front of the first car at $x=0-$ improves as $t$ increases. This induces the driver to speed up to $v_+$, either through a scanning rarefaction wave, or a scanning to acceleration shock, and then keep the speed $v_+$ with a more comfortable spacing $u_3$.
Both solution types \eqref{RSolver1} and \eqref{RSolver2} are possible, depending on the driver's actions.
Comparing with  \eqref{RSolver1}, solutions of shape \eqref{RSolver2} certainly require more driver's effort and more cost in fuel and car maintenance, without any saving in travel time. If the driver for the car at $x=0-$ is rational, he should act according to \eqref{RSolver1}. However, there is nothing stopping him to act like \eqref{RSolver2}, either by error or intention.

\section{Stop-and-go patterns and other solution examples}\label{s:examples}

In this section we show, by examples, solutions of initial value problems of the system \eqref{system2h} to see its capability of exhibiting several well-known traffic flow phenomena, emphasizing on those where LWR model fails. 
Solutions are constructed using Riemann solvers constructed in Section \ref{s:Riemann-solvers}, unless otherwise stated. Such solutions provide us bases to compare with experimental observations for validation purposes, and to point out the directions for improving the model later. To simplify notations, in the following we often identify states $(u_1, v_1)$, $(u_2, v_2)$ with circled numbers as $\circled{1}$, $\circled{2}$ and so on.

\begin{example}[Car train with stationary but varying spacing]\label{carTrain}
A car train is defined as a platoon of cars traveling at the same constant speed $v(x, t) \equiv v_0$, but the spacing $u(x, t) = u_0(x)$ does not have to be uniform in $x$. This is the result of different driving tastes (or hysteresis moods) for spacing.
It is easy to see that $(u, v)(x, t) = (u_0(x), v_0)$ is a solution of \eqref{system2h}, as long as 
$(u_0(x), v_0)\in \Omega$ for all $x\in \R$.
In contrast, classical macroscopic models, such as the LWR model, cannot have this kind of solution; 
ARZ model \cite{Aw-Rascle, Zhang2002} allows predetermined drivers' preferences for spacing, which do not change according to drivers' history.

\end{example}

\begin{example}[Stop-and-go solutions]\label{stopAndGo}
	Pick any two points $\circled{1} = (u_1,h_1)$ and $\circled{2} = (u_2,h_2)$ shown in Figure \ref{f:SNGS} on the left, such that $u_1 = u^D(h_1)$,
	$u_2 = u^A(h_2)$, and both chord conditions
	\eqref{chord_S2A} and \eqref{chord_S2D} are assumed to hold. 	This, according to Remark \ref{S2D and S2A}, can happen only in congested zone.


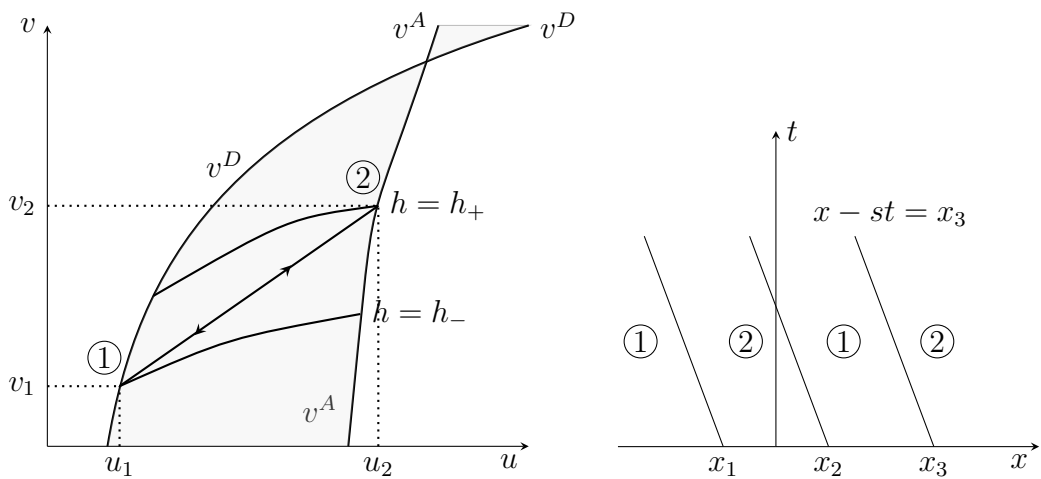
\begin{figure}[htbp]
\begin{tabular}{cc}
\begin{tikzpicture}[>=stealth, scale=0.8]

\draw[->] (0,0) -- (8,0) node[below=0.2cm, left] {$u$} coordinate (x axis);
\draw[->] (0,0) -- (0,7) node[left] {$v$} coordinate (y axis);

\draw[thick] (1,0)  .. controls (1.5,3) and (3,5.5) .. node[above=0.5cm,left=-0.5cm]{$v^D$} (8,7) node[right]{$v^D$};

\draw[thick] (6.5,7) node[left]{$v^A$} .. controls (5.2,3)  and (5.5,5) .. (5,0) node[above=0.5cm, left]{$v^A$};

\filldraw[fill=gray!20!white, draw=black, nearly transparent] (1,0) .. controls (1.5,3) and (3,5.5) ..  (8,7) -- (6.5,7) .. controls (5.2,3) and (5.5,5) .. (5,0) -- cycle;

\draw[thick] (1.76,2.5)  .. controls (4,3.8) .. (5.5,4) node[left=0.2cm,above]{$\circled{2}$} node[right]{$h=h_+$};
\draw[thick,dotted] (5.5,0) node[below]{$u_2$}-- (5.5,4);
\draw[thick,dotted] (0,4) node[left]{$v_2$}-- (5.5,4);
\draw[thick] (1.2,1) node[left=0.2cm,above]{$\circled{1}$} .. controls (3,1.8) .. (5.2,2.2) node[right]{$h=h_-$};
\draw[thick,dotted] (1.2,0) node[below]{$u_1$}-- (1.2,1);
\draw[thick,dotted] (0,1) node[left]{$v_1$}-- (1.2,1);
\draw[thick] (1.2,1) -- (5.5,4);
\draw[thick,->] (2.5,1.9) -- (2.45,1.85); 
\draw[thick,->] (4,2.96) -- (4.05,3.01); 
\end{tikzpicture}

&
\begin{tikzpicture}[>=stealth, scale=0.7]

\draw[->] (0,0) -- (5,0) node[below=0.2cm, left] {$x$} coordinate (x axis);
\draw (0,0) -- (-3,0);
\draw[->] (0,0) -- (0,6) node[right] {$t$} coordinate (y axis);

\draw (3,0) node[below]{$x_3$}  -- node[midway,left=0.3cm]{$\circled{1}$}
node[midway,right=0.2cm]{$\circled{2}$}
(1.5,4) node[above=0.3cm,right=-0.7cm]{$x-st=x_3$};

\draw (1,0) node[below]{$x_2$} -- node[midway,left=0.2cm]{$\circled{2}$} (-0.5,4) node[above=1cm,right=0.5cm]{};
\draw (-1,0) node[below]{$x_1$} -- node[midway,left=0.2cm]{${\circled{1}}$} (-2.5,4) node[above=0.3cm,left=-1.8cm]{};

\end{tikzpicture}
\end{tabular}
\caption{{Stop-and-go solutions: left, in the $(u,v)$-plane; right, in the $(x,t)$-plane. }}
\label{f:SNGS}
\end{figure}
	
Theorem \ref{t:scz1} states that there is a shock wave from $(u_-,v_-)=(u_1,v_1)$ to $(u_+,v_+)=(u_2,v_2)$. On the other hand, Theorem \ref{t:scz2} says that there is another shock wave connecting $(u_-,v_-)=(u_2,v_2)$ to $(u_+,v_+)=(u_1,v_1)$. Furthermore, these two shock waves have the same speed $s<0$. This gives us stop-and-go solutions. For instance, the function $(u,v)=\left(u(x,t),v(x,t)\right)$ shown in Figure \ref{f:SNGS} on the right,
	is a solution of the inviscid system \eqref{system2h}. The speeds of cars at any given time $t$ alternate as $x$ increases, from slow (with velocity $v_1$) to fast (with velocity $v_2$). Furthermore, all boundaries of speed alternation travel at the same speed $s<0$ so that the length of each speed zone does not change.

In \cite{Yuan-Knoop-Hoogendoorn}, the authors observed empirically stop-and-go waves in actual traffic, and found that the speed in congestion ($v_1$) and that at the outflow of congestion ($v_2$) can range broadly, from 6.3 to 48.7 km/h for $v_1$ and from 29.3 to 61.2 km/h  for $v_2$.
	The solution construction shown above indeed allows wider range of choices for $v_1$ and $v_2$ as long as they differ sufficiently in the congested zone.

We emphasize that the construction of such stop-and-go solutions is not possible if either $u_1$ or $u_2$ are in free zone $\{u>u_c\}$.
Thus, the model \eqref{system2h} predicts that there is no stop-and-go waves in free zone.
\end{example}

\begin{example}[Sufficient speed variation in a car train in congested zone caused by a temporary moving bottleneck can create stop-and-go waves]\label{ex:temp-bottleneck}
	Consider a car train with constant spacing and speed $(\bar u, \bar{v})$ in congested zone. Assume $(\bar u, \bar{v})$
	is in the interior of $\Omega$, as shown in Figure \ref{f:Bottleneck} on the left.	
\begin{figure}[htbp]
\begin{tabular}{cc}
\begin{tikzpicture}[>=stealth, scale=0.8]

\draw[->] (0,0) -- (8,0) node[below=0.2cm, left] {$u$} coordinate (x axis);
\draw[->] (0,0) -- (0,7) node[left] {$v$} coordinate (y axis);

\draw[thick] (1,0) node[above=0.5cm, left]{$v^D$} .. controls (1.5,3) and (3,5.5) .. (8,7) node[right]{$v^D$};

\draw[thick] (6.5,7) node[left]{$v^A$} .. controls (5.2,3)  and (5.5,5) .. (5,0) node[above=0.5cm, left]{$v^A$};

\filldraw[fill=gray!20!white, draw=black, nearly transparent] (1,0) .. controls (1.5,3) and (3,5.5) ..  (8,7) -- (6.5,7) .. controls (5.2,3) and (5.5,5) .. (5,0) -- cycle;

\draw[thick] (2.2,3.3)  .. controls (4,3.8) .. (5.5,4) node[above=0.1cm,right=0.1cm]{$h=\bar h$} node[below=0.3cm,right=-0.1cm]{$\circled{4}$};

\draw[thick] (1.42,1.7)  .. controls (3.5,2.4) .. (5.25,2.7) node[right]{$h=h_1$};

\path [draw=black,thick,postaction={on each segment={mid arrow=black,thick}}]
(3.8,3.72) --  (1.75,2.45) node[left=0.2cm,above=-0.1cm]{$\circled{3}$}-- (3.8,2.45);
\fill (3.8,3.72) circle (2pt) node[above]{$\circled{1}$};
\draw[thick,dotted] (3.8,0) node[below]{$\bar u$}-- (3.8,3.72);
\draw[thick,dotted] (0,3.72) node[left]{$\bar{v}$}-- (3.8,3.72);
\draw[thick,dotted] (1.75,0) node[below]{$u_2$}-- (1.75,2.45);
\draw[thick,dotted] (0,2.45) node[left]{$v_1$}-- (1.75,2.45);

\path [draw=black,thick,postaction={on each segment={mid arrow=black,thick}}]
(3.8,2.45) -- (5.4,3.72) -- (3.8,3.72);
\fill (3.8,2.45) circle (2pt) node[below=0.3cm,right=0cm]{$\circled{2}$};
\draw[thick,dotted] (5.4,0) node[below]{$u_3$}-- (5.4,3.72);

\path [draw=black,postaction={on each segment={mid arrow=black,thick}}]
(1.75,2.45) -- (5.4,3.72);

\end{tikzpicture}


\begin{tikzpicture}[>=stealth, scale=0.8]


\draw[->] (0,0) node[below]{$0$} -- (2,0) node[below=0.2cm, left] {$x$} coordinate (x axis);
\draw (0,0) -- (-6,0);
\draw[->] (0,0) -- (0,4) node[left] {$t$} coordinate (y axis);

\draw (-3,0) node[left=1.5cm, above=0.1cm]{$\circled{1}$} --  (-6,2.5); 
\draw (-3,0) node[below]{$-1$}-- (-3,3); 
\draw (-3,3) node[left=0.8cm]{$\circled{3}$} node[right=1.3cm, below=0cm]{$\circled{4}$}--  (-3.5,4);  
\draw (0,0) node[left=1.5cm, above=0.1cm]{$\circled{2}$} node[right=1cm, above=0.1cm]{$\circled{1}$}-- (-3,3) ; 

\draw[thick, dotted] (0,3) node[right]{$t_1$} -- (-3,3);

\end{tikzpicture}

\end{tabular}
\caption{{Formation of a stop-and-go pattern. The figure on the left is in the $(u,v)$-plane, where thick lines are waves produced at $t=0$ and thin lines are the shock waves emitted at $t=t_1$. The figure on the right shows the solution in the $(x,t)$-plane.}}
\label{f:Bottleneck}
\end{figure}
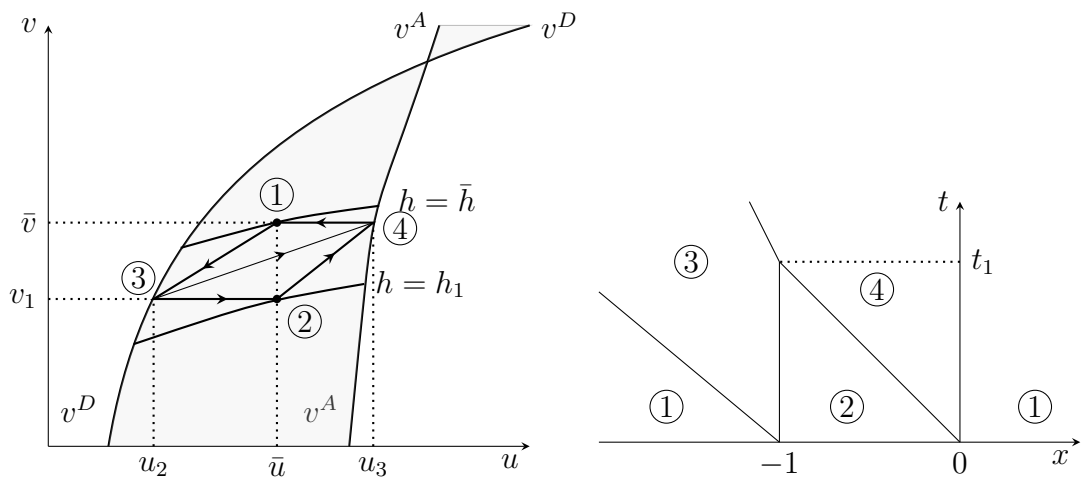
What is the effect if cars in the range $-1<x<0$ brake to a lower speed $v_1< \bar{v}$,  due to a temporary bottleneck at $x=0$?
Such temporary bottleneck can be caused by, for instance, a car insertion at $x=0$
due to lane change, or just because of a driver's random behavior.
To simulate this situation, consider the initial condition
\begin{equation}\label{SnGoInitialData}
(u,v)(x,0)=\left\{
\begin{array}{ll}
	(\bar u, \bar{v}) = \circled{1} & \hbox{ if }x<-1,
	\\
	(\bar u, v_1) = \circled{2} & \hbox{ if } -1\le x \le 0,
	\\
	(\bar u, \bar{v})= \circled{1} & \hbox{ if } x>0,
	\end{array}
	\right.
	\end{equation}
where both $(\bar u,\bar{v})$ and $(\bar u,v_1)$ are in the interior of the scanning region. All the shock connections in Figure \ref{f:Bottleneck} exist if either the scanning curves $h=h_1$ and $h = \bar h$ are sufficiently flat or $v_1$ is low enough with respect to $\bar{v}$. The initial jumps at $x=-1$ and $x=0$ are solved by (see cases 2{\em (ii)} and 2{\em (iii)(c)} in Section \ref{s:Riemann-solvers}) respectively as
	\[
	\circled{1} \xrightarrow{\text{ScDS}} \circled{3}
	\xrightarrow{\text{ST}} \circled{2},
\qquad
	\circled{2} \xrightarrow{\text{ScAS}} \circled{4}
	\xrightarrow{\text{ST}} \circled{1}.
	\]
	At time $t=t_1$, the stationary shock
	$\circled{3}
	\xrightarrow{\text{ST}} \circled{2}$ meets the scanning-to-acceleration shock
	$\circled{2} \xrightarrow{\text{ScAS}} \circled{4}$
	at $x=-1$. This interaction is solved by just one scanning-to-acceleration shock
	$\circled{3} \xrightarrow{\text{ScAS}}\circled{4}$, see Figure \ref{f:Bottleneck} on the right.
	This scanning-to-acceleration shock's speed is slower than that of
	$\circled{1} \xrightarrow{\text{ScDS}} \circled{3}$, and hence there is no wave interaction
	after $t=t_1$. The stop-and-go speed oscillation $\bar{v}\to v_1\to \bar{v}$ is then generated and will  persist as $t$ increases.
	Thus, the model \eqref{system2h} shows that phantom jam can be created in a uniform car train in congested zone by a few cars' random braking that create large enough speed variation.
	
	Similarly, if we let $v_1>\bar{v}$ sufficiently in congested zone, wave patterns similar to those in Figure \ref{f:Bottleneck} still emerge.
\end{example}

\begin{example}[Formation of stop-and-go waves on a closed circular track]\label{circularTrack}
To show that stop-and-go waves can emerge without any boundary influence such as a road bottleneck at a fixed location,
	Sugiyama et al. \cite{Sugiyama2004} did experiments in which cars travel in a closed loop road.
	They showed that car trains uniform at initial time
	developed stop-and-go patterns later, and a jam region  propagated upstream through otherwise faster moving cars.
	Can the model \eqref{system2h} exhibit the formation of such pattern?
	The answer is affirmative.
	To see this, we continue Example \ref{ex:temp-bottleneck} with the modification that the road is a circle. Notice that Lagrange coordinate $x$ is label of cars.  The car platoon on a circular road must satisfy the periodic boundary condition
	\begin{equation}\label{periodicCondition}
	(u, v)(x+L, t) = (u, v)(x, t),
	\end{equation}
	where $L>2$ is the number of cars in the circular lane.
	The initial condition \eqref{SnGoInitialData}
	is modified as
	\begin{equation}\label{SnGoInitialData2}
	(u,v)(x+nL,0) = (u,v)(x,0)=\left\{
	\begin{array}{ll}
	(\bar u, \bar v) = \circled{1} & \hbox{ if } -L/2\le x<-1,
	\\
	(\bar u, v_1) = \circled{2} & \hbox{ if } -1\le x \le 0,
	\\
	(\bar u, \bar v) = \circled{1} & \hbox{ if } 0 < x < L/2,
	\\
	\end{array}
	\right.
	\end{equation}
	for $n=\pm 1, \pm 2, \ldots$.
	Figure \ref{f:Bottleneck} on the right is still correct if $L$ is large enough.
	As $t$ increases, the shock
	$\circled{1} \xrightarrow{\text{ScDS}} \circled{3}$
	will arrive at $x=-L$, equivalent to $x=0$ on the circular lane, at time $t=t_2$ before any other shock arrives at $x=0$.
	Then, at $t=t_2+$, the cars on the circular track are divided into two zones: one is the $\circled{3} = (u_2, v_1)$ zone in which the traffic is denser and slower, the other is the $\circled{4} = (u_3, \bar v)$ zone in which the traffic is sparser and faster. This wave pattern is persistent because there is no more wave interaction after $t_2$.
	
	If $v_1$ is moderately smaller than $\bar v$, then at fixed time right after $t_2+$, the wave pattern in $v$ is, from upstream down, a slowdown shock and a rarefaction speedup wave. These regions travel upstream. As $t$ further increases, more wave interactions will occur. Our numerical simulations using an upwinding scheme show that the wave pattern for $v$ remains, but the variation of $v$ decay as $t$ increases.
	
	From Figure \ref{f:Bottleneck}, it is clear that if the scanning curve slope is small, then it takes less deviation in speed, $\bar v - v_1$, to produce a stop-and-go pattern.
	
\end{example}

\begin{example}[Effect of a few under-speeding cars on an otherwise uniform car train in free zone $u>u_c$ is temporary.]
Example \ref{ex:temp-bottleneck} shows that a large enough speed disturbance in an otherwise uniform car train traveling in congested zone $u<u_c$ can generate a persistent stop-and-go wave. It is then natural to investigate the effect of a large enough temporary bottleneck on an otherwise uniform car train traveling in free zone $u>u_c$.To mimic such situation, consider the initial data shown in Figure \ref{f:Bottleneck-freezone}, representing an otherwise uniform car train traveling with state $\circled{1}$, having a few cars that travels temporarily at a much slower state $\circled{2}$. The solution constructed in Figure \ref{f:Bottleneck-freezone} on the right shows that after some time, all cars travels in either state $\circled{1}$ or $\circled{3}$, all with original fast speed $v_1$. The effect of these few slow cars at initial time is eliminated in a finite time. Notice that the wave from $\circled{2}$ to $\circled{3}$ is composed by a deceleration-to-acceleration shock glued on the left to an acceleration rarefaction wave.
	

\begin{figure}[htbp]
\begin{tabular}{cc}
\begin{tikzpicture}[>=stealth,scale=0.7]


\draw[->] (0,0) -- (12,0) node[below=0.2cm, left] {$u$} coordinate (x axis);
\draw[->] (0,0) -- (0,8) node[left] {$v$} coordinate (y axis);

\draw (0,7.5) node[left] {$\bar v$} --  (11,7.5);

\draw[thick] (1,0) node[above=0.5cm, left]{$v^D$} .. controls (1.5,3) and (3,5.5) .. (11,6.5) node[right]{$v^D$};

\draw[thick] (11,7.3) node[right]{$v^A$} .. controls (5,7)  and (4,7.1) .. (3.5,0) node[above=0.5cm, right]{$v^A$};

\filldraw[fill=gray!20!white, draw=black, nearly transparent] (1,0) .. controls (1.5,3) and (3,5.5) ..  (11,6.5) -- (11,7.3) .. controls (5,7) and (4,7.1) .. (3.5,0) -- cycle;

\path [draw=black,thick,postaction={on each segment={mid arrow=black,thick}}]
(10,6.35) node[below=0.2cm,right=0cm]{$\circled{1}$} -- (2.95,3.8) node[above]{$\circled{2}$}-- (4.97,6);
\draw[thick,->] (5,6) -- (5.05, 6.05);
\draw[thick,dotted] (10,0) node[below]{$u_1$} -- (10,6.35);
\draw[thick,dotted] (0,6.35) node[left]{$v_1$} -- (10,6.35);
\fill (10,6.35) circle (2pt);

\draw[thick,dotted] (2.95,0) node[below]{$u_2$} -- (2.95,3.8);
\draw[thick,dotted] (0,3.8) node[left]{$v_2$} -- (2.95,3.8);
\fill (2.95,3.8) circle (2pt);

\path [draw=black,thick,postaction={on each segment={mid arrow=black,thick}}]
(5.4,6.35) node[above]{$\circled{3}$} -- (10,6.35);
\draw[thick,dotted] (5.4,0) node[below]{$u_3$} -- (5.4,6.35);
\fill (5.4,6.35) circle (2pt);

\end{tikzpicture}
\begin{tikzpicture}[>=stealth, scale=0.7]

\draw[->] (0,0) node[below]{$0$} node[above=0.3cm, right=0cm]{$\circled{1}$} -- (1,0) node[below=0.2cm, left] {$x$} coordinate (x axis);
\draw (0,0) node[above=1cm, left=0.2100cm]{$\circled{3}$} -- (-6,0);
\draw[->] (0,0) -- (0,6) node[left] {$t$} coordinate (y axis);

\draw (0,0) -- (-5,2.6);
\draw (0,0) -- (-5.65,3.3);
\draw (0,0) -- (-6,4);
\draw (-2,0)
node[below]{$-1$}
node[above=0.3cm, right=-0.2cm]{$\circled{2}$}
node[above=0.3cm, left=0.8cm]{$\circled{1}$}
-- (-5,2.6); 

\draw (-5,2.6) .. controls (-5.5,3) .. (-6,4);

\draw (-6,4) -- (-6,5);
\draw[thick,dotted] (0,4) node[right]{$t_1$}-- (-6, 4);
\end{tikzpicture}
\end{tabular}
\caption{A few temporarily slow cars in an otherwise uniform car train in the free zone have no effect on car platoon's speed after time $t_1$.}
\label{f:Bottleneck-freezone}
\end{figure}
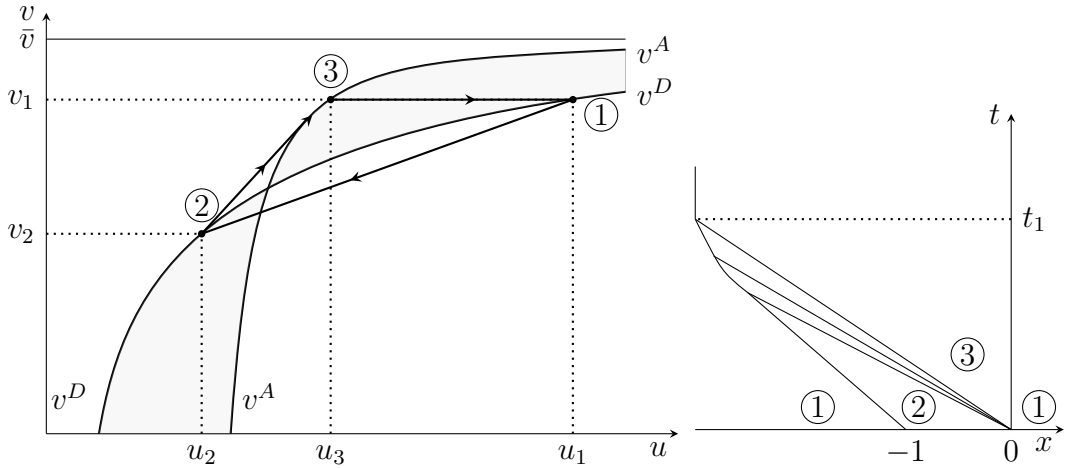

	
\end{example}

\begin{example}[Decay of a stop-and-go wave due to sparser vehicles upstream or downstream]
	Let us continue to consider Example \ref{stopAndGo} and Figure \ref{f:SNGS}.
	Intuitively, the deceleration wave will decay if the car platoon {\em upstream} is sparse enough, so that the upstream state is more ``compressible" and can absorb the deceleration wave. To test whether the model \eqref{system2h} agrees with this observation, we set up the initial data as
	\begin{equation}\label{SnGoDecayInitialData}
	(u,v)(x,0)=\left\{
	\begin{array}{ll}
	(u_3,  v_3) = \circled{3}  & \hbox{ if }x\le -2,
	\\
	(u_1, v_1) = \circled{1}  & \hbox{ if }-2<x<-1,
	\\
	(u_2, v_2) = \circled{2} & \hbox{ if } -1\le x \le 0,
	\\
	(u_1, v_1) = \circled{1} & \hbox{ if } x>0,
	\end{array}
	\right.
	\end{equation}
	where $u_3>> u_2$. For convenience, select $v_3 = v^D(u_3)$, see Figure \ref{f:SNGupdown}. This simulates the situation where there is a stop-and-go wave moving upstream (modeled in the initial data through the sequence $\circled{1}\,\circled{2}\,\circled{1}$) and the traffic upstream is in free zone with spacing $u_3$ sparse enough. In such a case, the slope of the line connecting states $\circled{3}$ and $\circled{4}$ is small enough such that the shock connecting {\circled 1} and {\circled 2} will eventually interact; this eliminates the slow region with slow speed $v_2$ after time $t_3$. The solution of \eqref{system2h} in $(x,t)$-plane is depicted in Figure \ref{f:SNGupdown} on the right.

\begin{figure}[htbp]
\begin{tabular}{cc}
\begin{tikzpicture}[>=stealth,scale=0.7]
\draw[->] (0,0) -- (12,0) node[below=0.2cm, left] {$u$} coordinate (x axis);
\draw[->] (0,0) -- (0,8) node[left] {$v$} coordinate (y axis);

\draw[thick] (1,0) node[above=0.5cm, left]{$v^D$} .. controls (1.3,4) and (3,5.5) .. (11,6.8) node[right]{$v^D$};

\draw[thick] (11,7.3) node[right]{$v^A$} .. controls (8,7)  and (4.9,7.1) .. (4,0) node[above=0.5cm, right]{$v^A$};

\filldraw[fill=gray!20!white, draw=black, nearly transparent] (1,0) .. controls (1.3,4) and (3,5.5) ..  (11,6.8) -- (11,7.3) .. controls (8,7) and (4.9,7.1) .. (4,0) -- cycle;

\draw[thick] (10,6.62) node[below=0.2cm, right=0cm]{\circled{3}}-- (1.1,1); 
\draw[thick,->] (6.05,4.13) -- (6,4.1); 
\draw[thick, dotted] (10,0) node[below]{$u_3$} -- (10,6.62);
\fill (10,6.62) circle (2pt);

\draw[thick, dotted] (0,6.62) node[left]{$v_3$} -- (10,6.62);
\draw[thick, dotted] (1.1,0) node[below]{$u_2$} -- (1.1,1);
\fill (1.1,1) circle (2pt);

\draw[thick, dotted] (0,1) node[left]{$v_2$} -- (1.1,1);
\draw (1.1,1)  node[left=0.3cm,above=0.1cm]{$\circled{2}$};

\draw[thick] (1.1,1) -- (5,4); 
\draw[thick,->] (2.5,2.1) -- (2.45,2.07); 
\draw[thick,->] (3.45,2.82) -- (3.5,2.86); 
\path [draw=black,thick,postaction={on each segment={mid arrow=black,thick}}]
(2.6,4) -- (1.1,1);

\draw[thick] (10,6.62) -- (2.6,4); 
\draw[thick,->] (6.05,5.22) -- (6,5.2); 

\path [draw=black,thick,postaction={on each segment={mid arrow=black,thick}}] 
(2.6,4) node[above]{$\circled{4}$} -- (5,4);
\draw[thick, dotted] (5,0) node[below]{$u_1$} -- (5,4);
\fill (5,4) circle (2pt);
\fill (2.6,4) circle (2pt);
\draw (5,4) node[right=0.5cm,above=-0.1cm]{$\circled{1}$};
\draw[thick, dotted] (0,4) node[left]{$v_1$}-- (5,4);

\end{tikzpicture}

\begin{tikzpicture}[>=stealth, scale=0.7]

\draw[->] (0,0) node[below]{$0$} node[above=0.3cm, right=0cm]{$\circled{1}$} -- (1,0) node[below=0.2cm, left] {$x$} coordinate (x axis);
\draw (0,0) -- (-6,0);
\draw[->] (0,0) -- (0,8) node[left] {$t$} coordinate (y axis);

\draw (0,0) -- (-6,6); 
\draw (-2,0) node[below]{$-1$} node[above=0.3cm, right=0.2cm]{$\circled{2}$} -- (-4,2); 
\draw (-4,0) node[below]{$-2$}
node[above=0.3cm, right=0.2cm]{$\circled{1}$}  -- (-4,2)
node[below=0.3cm, left=-0.1cm]{$\circled{4}$} ; 
\draw (-4,0) node[above=0.3cm, left=0.5cm]{$\circled{3}$}  -- (-5.3,2.5); 
\draw (-4,2) -- (-5.3,2.5); 
\draw (-5.3,2.5) -- (-6,6); 
\draw (-6,6) --(-6,7.5); 
\draw (-6,6) --(-6.5,7.5) node[above=0.3cm,right=-0.2cm]{$\circled{4}$}; 

\draw[thick,dotted] (0,2) node[right]{$t_1$}-- (-4,2);
\draw[thick,dotted] (0,2.5) node[right]{$t_2$}-- (-5.3,2.5);
\draw[thick,dotted] (0,6) node[right]{$t_3$}-- (-6,6);
\end{tikzpicture}
\end{tabular}
\caption{{Decay of a stop-and-go wave due to sparser vehicles upstream.}}
\label{f:SNGupdown}
\end{figure}
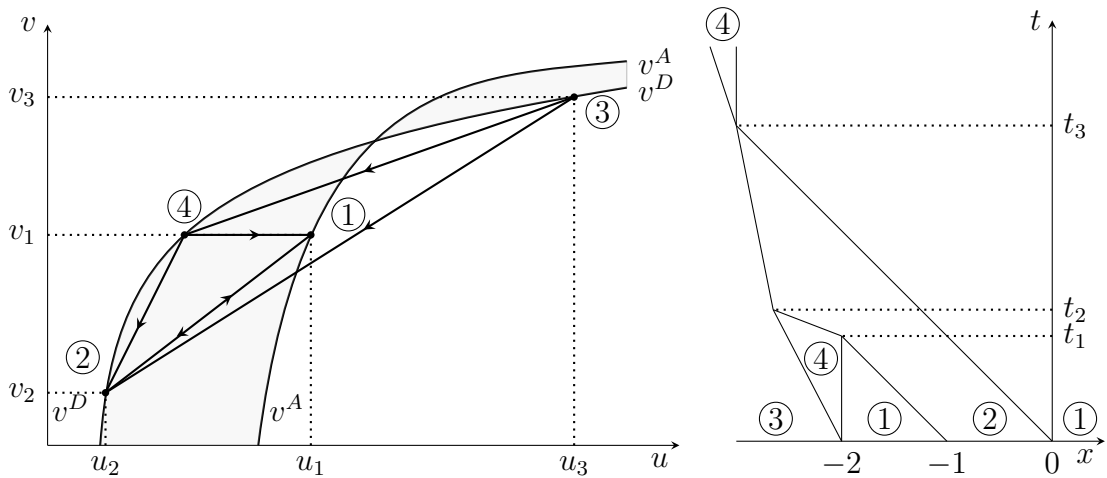


Another way for the stop-and-go waves, as is for any traffic jam, to disappear in time is that the {\em downstream} traffic travels faster. To confirm that the model \eqref{system2h} exhibits this phenomenon, we set the initial data as is shown in Figure \ref{f:SNGdisappearD} on the right. The downstream traffic is faster, with $v_3 > v_2$. Choose $v_3 = v^A(u_3)$ for simplicity. The solution depicted in Figure \ref{f:SNGdisappearD} shows that, in a finite time, the slow region with speed $v_2$ is eliminated and every car's speed will eventually increase to $v_3$.
\begin{figure}[htbp]
\begin{tabular}{cc}
\begin{tikzpicture}[>=stealth, scale=0.8]

\draw[->] (0,0) -- (8,0) node[below=0.2cm, left] {$u$} coordinate (x axis);
\draw[->] (0,0) -- (0,7) node[left] {$v$} coordinate (y axis);

\draw[thick] (1,0) node[above=0.5cm, left]{$v^D$} .. controls (1.5,3) and (3,5.5) .. (8,7) node[right]{$v^D$};

\draw[thick] (6.5,7) node[left]{$v^A$} .. controls (5.2,3)  and (5.5,5) .. (5,0) node[above=0.5cm, left]{$v^A$};

\filldraw[fill=gray!20!white, draw=black, nearly transparent] (1,0) .. controls (1.5,3) and (3,5.5) ..  (8,7) -- (6.5,7) .. controls (5.2,3) and (5.5,5) .. (5,0) -- cycle;

\draw[thick] (2.2,3.3)  .. controls (4,3.8) .. (5.5,4) node[right=0.3cm]{$h=h_1$}
node[left=0cm,above=-0.1cm]{$\circled{1}$};
\draw[thick,dotted] (5.5,0) node[below]{$u_1$} -- (5.5,4);
\draw[thick,dotted] (0,4) node[left]{$v_1$} -- (5.5,4);
\fill (5.5,4) circle (2pt);

\draw[thick] (1.42,1.7) node[left=0.2cm,above=0cm]{$\circled{2}$} .. controls (3.5,2.4) .. (5.25,2.7) node[right=0.5cm]{$h=h_2$};
\draw[thick,dotted] (1.42,0)  node[below]{$u_2$} -- (1.42,1.7);
\fill (1.42,1.7) circle (2pt);
\draw[thick,dotted] (0,1.7)  node[left]{$v_2$} -- (1.42,1.7) ;

\draw[thick] (3.35,4.6)  .. controls (5,5) .. (5.85,5.1) node[right]{$h=h_3$}
node[left=0.2cm,above=0cm]{$\circled{3}$};
\draw[thick,dotted] (5.85,0) node[below=0.3cm, right=-0.2cm]{$u_3$} -- (5.85,5.1);
\draw[thick,dotted] (0,5.1) node[left]{$v_3$} -- (5.85,5.1);
\fill (5.85,5.1) circle (2pt);
\draw[thick] (1.42,1.7) -- (5.5,4);
\draw[thick,->] (2.5,2.3) -- (2.45,2.27);
\draw[thick,->] (4,3.17) -- (4.05,3.20);
\draw[thick,->] (5.75,4.8) -- (5.8,4.9); 

\path [draw=black,thick,postaction={on each segment={mid arrow=black,thick}}]
(1.42,1.7) -- (5.85,5.1);
\end{tikzpicture}


\begin{tikzpicture}[>=stealth, scale=0.8]


\draw[->] (0,0) -- (2,0) node[below=0.2cm, left] {$x$} coordinate (x axis);
\draw (0,0) -- (-6,0);
\draw[->] (0,0) -- (0,6) node[left] {$t$} coordinate (y axis);

\draw (0,0)
node[left=0.7cm, above=1cm]{$\circled{3}$}
node[right=0.7cm, above=0.1cm]{$\circled{3}$}
node[left=1cm, above=0.1cm]{$\circled{2}$} --  (-3,3);
\draw (-2,0) node[left=0.7cm, above=0.1cm]{$\circled{1}$} --  (-3,3);
\draw (-3,3) -- (-5,4); 
\draw (-3,3) -- (-5,4.2);
\draw (-3,3) -- (-5,4.5);
\end{tikzpicture}

\end{tabular}
\caption{{The slow region in the middle of a stop-and-go wave disappears in time when the traffic  downstream is faster. 
}}
\label{f:SNGdisappearD}
\end{figure}
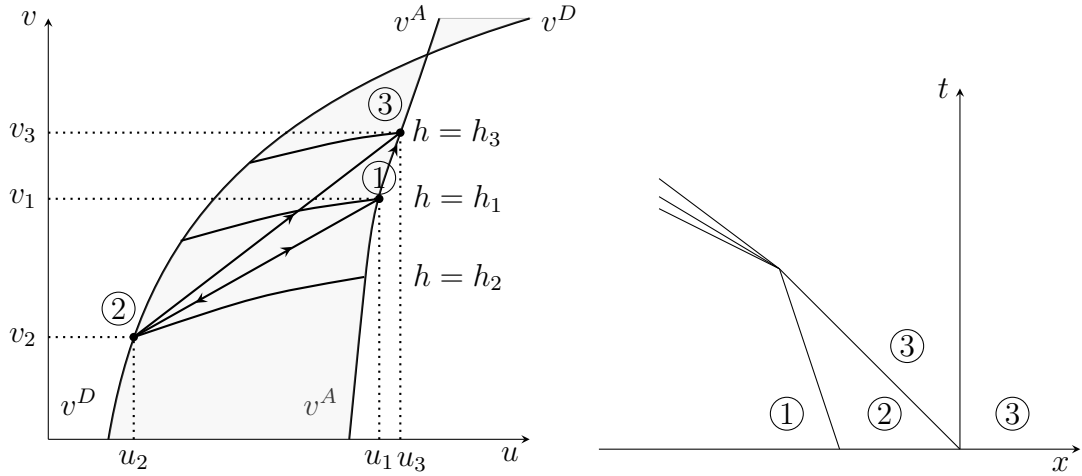


\end{example}

\begin{example}[Decay of phantom jam caused by small disturbance]\label{smalllPerturbation} Example \ref{ex:temp-bottleneck} shows that when a few cars in an otherwise uniform (in speed and spacing) platoon of cars deviate in speed by a sufficient margin, then stop-and-go patterns appear and persist as $t$ increases. In this example, we investigate what happens when the deviation is small. Notice that, differently from Example \ref{ex:temp-bottleneck}, cars never reach the acceleration and deceleration curves; the small disturbance is simulated by imposing that the traffic flow is strictly contained in the scanning zone.
	
	Assume $v_1<\bar v$ for definiteness; the arguments and results for the other case are similar.
	Consider the initial condition  given
	by \eqref{SnGoInitialData} with $v_1$ close enough to $\bar v$ so that
	$u_2> u^D(\bar h)$ and $u_3< u^A(h_1)$; see Figure \ref{f:platoon-dev} on the left. This requires $\circled{1} = (\bar u, \bar v)$ and
	$\circled{2} = (\bar u, v_1)$ to be away from the curves $v=v^A(u)$ and $v=v^D(u)$.
	

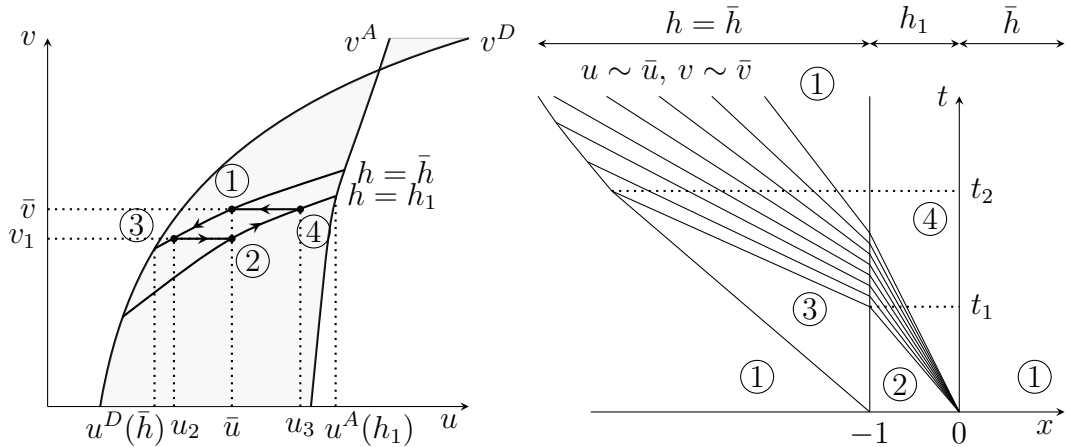
\begin{figure}[htbp]

\begin{tabular}{cc}

\begin{tikzpicture}[>=stealth, scale=0.70]


\draw[->] (0,0) -- (8,0) node[below=0.2cm, left] {$u$} coordinate (x axis);
\draw[->] (0,0) -- (0,7) node[left] {$v$} coordinate (y axis);

\draw[thick] (1,0) .. controls (1.5,3) and (3,5.5) .. (8,7) node[right]{$v^D$};

\draw[thick] (6.5,7) node[left]{$v^A$} .. controls (5.2,3)  and (5.5,5) .. (5,0);

\filldraw[fill=gray!20!white, draw=black, nearly transparent] (1,0) .. controls (1.5,3) and (3,5.5) ..  (8,7) -- (6.5,7) .. controls (5.2,3) and (5.5,5) .. (5,0) -- cycle;

\draw[thick] (2.03,3)  .. controls (3.5,3.8) .. (5.65,4.5) node[right]{$h=\bar h$};
\draw[thick, dotted] (2.03,0) node[below=0.3cm,left=-0.3cm]{$u^D(\bar h)$}-- (2.03,3);
\draw[thick,->] (2.8,3.42) -- (2.75,3.39); 
\path [draw=black,thick,postaction={on each segment={mid arrow=black,thick}}]
(4.8,3.75) node[below=0.3cm,right=-0.2cm]{${\circled{4}}$} -- (3.5,3.75) node[above]{${\circled{1}}$};
\fill (4.8,3.75) circle(2pt);
\draw[thick, dotted] (2.4,0) node[below=0.3cm,right=-0.2cm]{$u_2$} -- (2.4,3.19); 
\fill (2.4,3.19) circle(2pt);
\draw[thick,->] (4,3.46) -- (4.05,3.49); 
\draw[thick, dotted] (4.8,0) node[below]{$u_3$}-- (4.8,3.75); 

\draw[thick] (1.42,1.7)  .. controls (3.5,3.3) .. (5.47,4) node[right]{$h=h_1$};
\draw[thick, dotted] (5.47,0) node[below=0.3cm,right=-0.3cm]{$u^A(h_1)$}-- (5.47,4);
\path [draw=black,thick,postaction={on each segment={mid arrow=black,thick}}]
(2.4,3.19) node[left=0.5cm,above=-0.2cm]{$\circled{3}$} -- (3.5,3.19) node[below=0.3cm,right=-0.1cm]{${\circled{2}}$};

\fill (3.5,3.75) circle (2pt); 
\fill (3.5,3.19) circle (2pt); 
\draw[thick, dotted] (3.5,0) node[below]{$\bar u$}-- (3.5,3.7);
\draw[thick, dotted] (0,3.75) node[left]{$\bar v$}-- (3.5,3.75);
\draw[thick, dotted] (0,3.19) node[left]{$v_1$}-- (3.5,3.19);

\end{tikzpicture}


\begin{tikzpicture}[>=stealth, scale=0.70]


\draw[->] (0,0) -- (2,0) node[below=0.2cm, left] {$x$} coordinate (x axis);
\draw (0,0) -- (-7,0);
\draw[->] (0,0) node[below]{$0$} node[right=1cm,above=0.1cm]{$\circled{1}$} -- (0,6)
node[left] {$t$} coordinate (y axis);

\draw (-1.7,0)
node[left=1.5cm, above=0.1cm]{$\circled{1}$}
node[left=0.7cm, above=4cm]{$\circled{1}$}
node[left=2.7cm, above=4.2cm]{$u\sim\bar u,\, v\sim\bar v$}
node[left=0.8cm, above=1cm]{$\circled{3}$}
--  (-6.6,4.2); 
\draw (-6.6,4.2) .. controls (-7.6,5.4) .. (-8,6); 

\draw (-1.7,0)
node[below]{$-1$}
node[right=0.4cm, above=0cm]{$\circled{2}$}
node[right=0.8cm, above=2.2cm]{$\circled{4}$}
-- (-1.7,6);

\draw (0,0) -- (-1.7,2) ;
\draw (0,0) -- (-1.7,2.2) ;
\draw (0,0) -- (-1.7,2.4) ;
\draw (0,0) -- (-1.7,2.6) ;
\draw (0,0) -- (-1.7,2.8) ;
\draw (0,0) -- (-1.7,3);
\draw (0,0) -- (-1.7,3.2);
\draw (0,0) -- (-1.7,3.4) ;

\draw (-1.7,2) -- (-6.6,4.2);
\draw (-1.7,2.2) -- (-7.05,4.75);
\draw (-1.7,2.4) -- (-7.65,5.5);
\draw (-1.7,2.6) -- (-7.7,6);
\draw (-1.7,2.8) -- (-6.7,6);
\draw (-1.7,3) -- (-5.7,6);
\draw (-1.7,3.2) -- (-4.7,6);
\draw (-1.7,3.4) -- (-3.7,6) ;

\draw[thick, dotted] (-1.7,2) -- (0,2) node[right]{$t_1$};
\draw[thick, dotted] (-6.6,4.2) -- (0,4.2) node[right]{$t_2$};

\draw[<->] (-1.7,7) -- node[midway, above]{$h=\bar h$} (-8,7);
\draw[<->] (0,7) -- node[midway, above]{$h_1$} (-1.7,7);
\draw[<->] (0,7) -- node[midway, above]{$\bar h$} (2,7);
\end{tikzpicture}
\end{tabular}

\caption{{The disturbance of speed by a few cars slightly slower in an otherwise uniform platoon of cars will decay, but the sparser region created in front of these few cars is persistent. Left: in the $(u, v)$-plane; right: in the $(x, t)$-plane.}}
\label{f:platoon-dev}
\end{figure}

	The solution constructed using the Riemann solvers listed in Case 2 in Section \ref{s:Riemann-solvers}
	is illustrated in Figure \ref{f:platoon-dev} on the right.
	The interaction of the scanning rarefaction wave
	$\circled{2} \xrightarrow{\text{ScR}} \circled{4}$
	and the stationary shock
	$\circled{3} \xrightarrow{\text{ST}} \circled{2}$
	shifts the rarefaction wave along the $h=h_1$ scanning curve to that of $h=\bar h$.
	Later, this rarefaction wave with $h=\bar h$ will interact with the scanning shock of the same $\bar h$ and will eliminate the slow region with state $\circled{3}$.	This interaction is due to the nonlinearity of the scanning curve. Further interaction of the rarefaction and the shock wave will make cars' states in the interaction zone go to $\circled{1}$ at the rate of $O(1)t^{-1}$.
	
	It is interesting to see that the final state of the car train has spacing almost equal to $\bar u$ except for cars in $-1<x<0$, where it is $u_3$ and wider than $\bar u$.
	This is intuitive in the sense that cars in $-1<x<0$, which are initially slower than cars in $x>0$, will have more space in front because it takes time for them to accelerate to $\bar v$, and stay at that speed afterwards.
	The cars initially behind $x=-1$ first have to slow down to speed $v_1$, and later accelerate back to $\bar v$. Since they change speed along the same scanning curve, their spacing also return to the original $\bar u$. This phenomenon is disallowed in LWR model because it specifies $v=v(u)$, making different spacing for the same speed impossible.
	
	Similarly, if the few cars in the otherwise uniform car train in congested region are faster instead, then the disturbance in speed will also decay, but the spacing for those faster cars at $t=0$ will be more compact while the spacing for other cars remains almost the same after some time.
\end{example}

\begin{remark}
	The last examples suggest that if a few cars in an otherwise uniform car train brake sufficiently relative to slopes of the scanning curves nearby, then a stop-and-go phantom jam is created and the jam will persist.
	On the other hand, if the slowing down or speeding up is minor, then the speed oscillation in the car platoon will decay in time, if the initial data are strictly in scanning region.
	When the deviation from the uniform speed is up, then the can train will become denser for those cars overspeeding at $t=0$.
	Else, the spacing for these cars become larger.
	
	As a conclusion, if drivers are more likely to overspeed a little than underspeed in an otherwise uniform car train, then the car train will become more and more compact, until the state $(u, v)$ reaches the curve $v=v^D(u)$ at which overspeed is no longer safe without a steep deceleration afterwards.
\end{remark}

\begin{example}[Effect of a few cars underspeed in an otherwise uniform car train traveling at $v=v^D(u)$ is persistent]\label{exa:underspeed}
	This time, we change the initial data \eqref{SnGoInitialData} to
	\begin{equation}\label{SnGoInitialData3}
	(u,v)(x,0)=\left\{
	\begin{array}{ll}
	(\bar u, \bar v) = \circled{1} & \hbox{ if }x<-1,
	\\
	(u_1, v_1) = \circled{2} & \hbox{ if } -1\le x \le 0,
	\\
	(\bar u, \bar v)= \circled{1}  & \hbox{ if } x>0,
	\end{array}
	\right.
	\end{equation}
	where both $(\bar u, \bar v)$ and $(u_1, v_1)$, with $v_1<\bar v$, are on $v=v^D(u)$; see Figure \ref{f:fewunderD}. With respect to the initial data \eqref{SnGoInitialData}, we changed $(\bar u, v_1)$ to $(u_1, v_1)$ so as to compare the behavior of solutions of model \eqref{system2h} and that of the classical LWR model with $v= v^D(u$). For such initial data, LWR model's solution will eliminate the $(u_1, v_1)$ region in a finite time: the jumps at $x=-1$ and $x=0$ are solved by a shock wave from $\circled{1}$ to $\circled{2}$ and by a rarefaction wave $\circled{2}$ to $\circled{1}$, which necessarily interact and then annihilate.
	
	The solution of \eqref{system2h} with initial data \eqref{SnGoInitialData3} is depicted in Figure \ref{f:fewunderD}. Notice that in our model the shock from $x=-1$ and the rarefaction from $x=0$ {\em do not interact}, differently than in the case of the LWR model. This is due to the fact the constitutive law speeds are different in the two cases. The solution has a persistent slower (with car speed $v_1$) and expanding region moving upstream. This conclusion stays the same even if we increase $u_1$ towards $\bar u$. The expansion speed of the slower region is roughly proportional to the difference of the average slope of $v=v^D(u)$ and $v=v^S(u, h)$ near the point $(u_1, v_1)$.
	
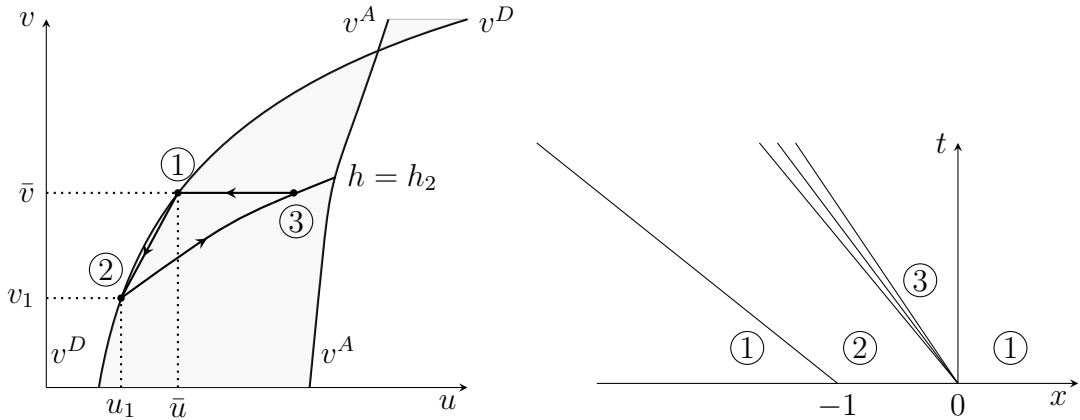
\begin{figure}[htbp]
\begin{tabular}{cc}
\begin{tikzpicture}[>=stealth, scale=0.7]

\draw[->] (0,0) -- (8,0) node[below=0.2cm, left] {$u$} coordinate (x axis);
\draw[->] (0,0) -- (0,7) node[left] {$v$} coordinate (y axis);

\draw[thick] (1,0) node[above=0.5cm, left]{$v^D$} .. controls (1.5,3) and (3,5.5) .. (8,7) node[right]{$v^D$};

\draw[thick] (6.5,7) node[left]{$v^A$} .. controls (5.2,3)  and (5.5,5) .. (5,0) node[above=0.5cm, right]{$v^A$};

\filldraw[fill=gray!20!white, draw=black, nearly transparent] (1,0) .. controls (1.5,3) and (3,5.5) ..  (8,7) -- (6.5,7) .. controls (5.2,3) and (5.5,5) .. (5,0) -- cycle;

\draw[thick] (1.42,1.7) node[left=0.2cm,above=0cm]{$\circled{2}$} .. controls (3.5,3.2) .. (5.5,4) node[right]{$h=h_2$};
\draw[thick,dotted] (1.42,0)  node[below]{$u_1$} -- (1.42,1.7) ;
\draw[thick,dotted] (0,1.7)  node[left]{$v_1$} -- (1.42,1.7) ;
\fill (1.42,1.7) circle(2pt);
\fill (4.7,3.7) circle(2pt);
\path [draw=black,thick,postaction={on each segment={mid arrow=black,thick}}]
(4.75,3.7) node[below]{$\circled{3}$} -- (2.5,3.7) node[above]{$\circled{1}$};

\path [draw=black,thick,postaction={on each segment={mid arrow=black,thick}}]
(2.5,3.7)--(1.42,1.7);

\draw[thick, dotted] (2.5,0) node[below]{$\bar u$}-- (2.5,3.7);
\fill (2.5,3.7) circle(2pt);
\draw[thick, dotted] (0,3.7) node[left]{$\bar v$} -- (2.5,3.7);

\draw[thick,->] (3,2.8) -- (3.05,2.83);

\end{tikzpicture}


\begin{tikzpicture}[>=stealth, scale=0.8]


\draw[->] (0,0) node[below]{$0$} -- (2,0) node[below=0.2cm, left] {$x$} coordinate (x axis);
\draw (0,0) -- (-6,0);
\draw[->] (0,0) -- (0,4) node[left] {$t$} coordinate (y axis);

\draw (0,0)
node[right=0.7cm, above=0.1cm]{$\circled{1}$}
node[left=1.3cm, above=0.1cm]{$\circled{2}$}
node[left=0.5cm, above=1cm]{$\circled{3}$}
--  (-3,4);
\draw (0,0) --  (-3.3,4);
\draw (0,0) --  (-2.7,4);

\draw (-2,0) node[below]{$-1$} node[left=1.2cm, above=0.1cm]{$\circled{1}$} --  (-7,4);

\end{tikzpicture}

\end{tabular}
\caption{{A few cars underspeed in an otherwise uniform car train traveling at $v=v^D(u)$ can create a persistent slower region.}}
\label{f:fewunderD}
\end{figure}

\end{example}

\begin{example}[Over-braking Riemann solvers \eqref{RSolver2} can create phantom jams]
	Consider the same initial data \eqref{SnGoInitialData}
	as in Examples \ref{ex:temp-bottleneck} and \ref{smalllPerturbation}, where $\circled{1}=(\bar u, \bar v)$ and $\circled{2} =(\bar u, v_1)$ are in the scanning region and
	$0< \bar v-v_1<<1.$
	Example \ref{smalllPerturbation} showed that the speed of all cars will equal $\bar v$ in finite time.
	The analysis used the Riemann solver \eqref{RSolver1}. On the other hand, we know from Section \ref{s:Riemann-solvers} that there are indefinitely many other Riemann solvers of the form
	\eqref{RSolver2} solving the same Riemann problem.
	The Riemann solver \eqref{RSolver2} is for the situation where a driver over-brakes, so we call \eqref{RSolver2} over-braking Riemann solver. An extreme case is when $\circled{1} = \circled{2}$; in this case \eqref{RSolver2} is still a solution besides the usual constant solution.
	
	What if \eqref{RSolver2} is used in Example \ref{smalllPerturbation}, instead of \eqref{RSolver1}?
	
	The initial value problem with data  \eqref{SnGoInitialData} is solved
	as shown in Figure \ref{f:DFR} on the right. 	

\begin{figure}[htbp]
\begin{tabular}{cc}

\begin{tikzpicture}[>=stealth, scale=0.65]


\draw[->] (0,0) -- (8,0) node[below=0.2cm, left] {$u$} coordinate (x axis);
\draw[->] (0,0) -- (0,7) node[left] {$v$} coordinate (y axis);

\draw[thick] (1,0) .. controls (1.5,4)  .. (3,7) node[left]{$v^D$};

\draw[thick] (8,7) node[right]{$v^A$} .. controls (7,5) .. (6,0);

\filldraw[fill=gray!20!white, draw=black, nearly transparent] (1,0) .. controls (1.5,4) ..  (3,7) -- (8,7) .. controls (7,5) .. (6,0) -- cycle;
%

\draw[thick] (1.43,3)  .. controls (3.5,4.9) .. (7.3,5.5) node[right]{$h=h_1$};
\draw[thick,dotted] (3,0) node[below]{$\bar u$}-- (3,4.3);
\fill (3,4.3) circle (2pt) node[above]{$\circled{1}$};
\draw[thick,dotted] (0,4.3) node[left]{$\bar v$}-- (3,4.3);
\draw[thick,dotted] (0,3.15) node[left]{$v_1$}-- (3,3.15);
\fill (3,3.15) circle (2pt) node[below=0.3cm,right=-0.1cm]{$\circled{2}$};

\draw[thick] (1.25,2)  .. controls (3.45,3.5) .. (6.95,4.5);
\draw[thin,->] (8,5) node[right]{$h=h_2$} -- (7.1,4.65);
\path [draw=black,thick,postaction={on each segment={mid arrow=black,thick}}]
(3,4.3) -- (1.13,1)  -- (6.65,3.15)  -- (3,3.15);
\fill (1.13,1) circle (2pt) node[left]{$\circled{3}$};
\fill (6.65,3.15) circle (2pt) node[right]{$\circled{4}$};

\path [draw=black,thick,postaction={on each segment={mid arrow=black,thick}}]
(6.25,4.3) -- (3,4.3);
\fill (6.25,4.3) circle (2pt) node[above]{$\circled{5}$};
\fill (6.92,4.3) circle (2pt) node[right]{$\circled{6}$};
\draw[thick,->] (4.5,3.78) -- (4.55,3.8);
\draw[thick, dotted] (6.25,4.3) -- (6.92,4.3);

\end{tikzpicture}


\begin{tikzpicture}[>=stealth, scale=0.65]


\draw[->] (0,0) -- (2,0) node[below=0.2cm, left] {$x$} coordinate (x axis);
\draw (0,0) -- (-7,0);
\draw[->] (0,0)
node[below]{$0$}
node[right=1cm,above=0.1cm]{$\circled{1}$}
-- (0,6)
node[right] {$t$} coordinate (y axis);

\draw (-2,0) -- (-8,3);
\draw (-2,0)
node[left=2.5cm, above=0.1cm]{$\circled{1}$}
node[left=1.5cm, above=3cm]{$\circled{6}$}
node[left=0.8cm, above=0.8cm]{$\circled{4}$}
--  (-5,2.5); 
\draw (-5,2.5) .. controls (-6.5,4) .. (-8,5); 

\draw (-2,0)
node[below]{$-1$}
node[right=0.4cm, above=0.1cm]{$\circled{2}$}
node[right=0.7cm, above=1.7cm]{$\circled{5}$}
-- (-2,6);

\draw (0,0) -- (-2,2);
\draw (0,0) -- (-2,2.2);
\draw (0,0) -- (-2,2.4) ;

\draw (-2,2) -- (-5,2.5);
\draw (-2,2.2) -- (-5.7,3.2);
\draw (-2,2.4) -- (-6.45,3.9) node[below=0.3cm]{$\circled{3}$};

\draw[thick, dotted] (-2,2) -- (0,2) node[right]{$t_1$};
\draw[thick, dotted] (-6.45,3.9) -- (0,3.9) node[right]{$t_2$};

\draw[<->] (-2,6.5) -- node[midway, above]{$h=\bar h$} (-8,6.5);
\draw[<->] (0,6.5) -- node[midway, above]{$h=h_1$} (-2,6.5);
\draw[<->] (0,6.5) -- node[midway, above]{$h=\bar h$} (2,6.5);

\end{tikzpicture}
\end{tabular}
\caption{{Example \ref{smalllPerturbation} re-done using over-braking in the Riemann solver \eqref{RSolver2}.}}
\label{f:DFR}
\end{figure}
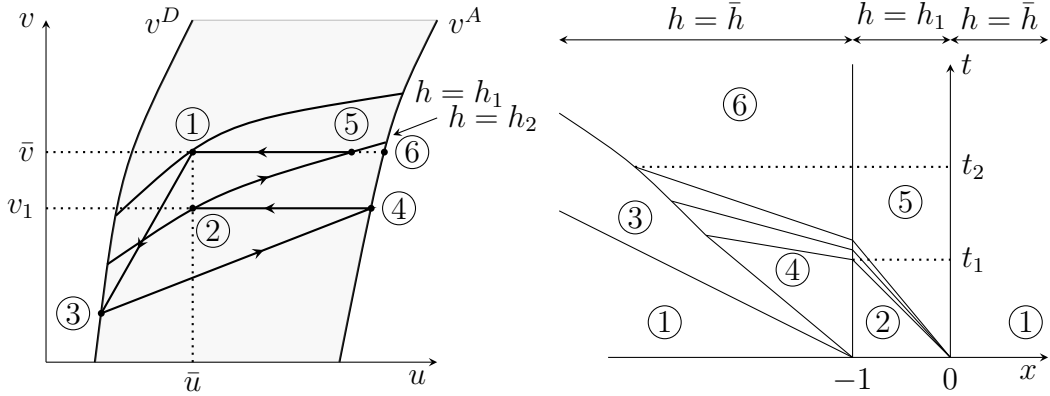

The Riemann solver used for the Riemann problem at $x=-1$ is of the type \eqref{RSolver2}. Notice that the two shock waves on the left do not interact. After waves finish their interactions in a finite time, the car platoon exhibits a phantom jam region $\circled{3}$ moving backwards along the car platoon, while the rest of cars travel at the higher speed $v_1$.
	
	On a closed loop road, similar reasonings as in Example \ref{circularTrack}
	show that the solution for $t$ large enough has only two regions in which cars' speeds are
	slow ($v_3$) and fast ($v_6$), respectively. These two regions have stationary length and are moving upstream at the speed of the shock
	$\circled{3} \xrightarrow{\text{ScAS}} \circled{6}$, and this wave pattern persist afterwards.
	This is again the stop-and-go wave pattern observed in
	\cite{Sugiyama2004}.
	
	Notice that the nonuniqueness of Riemann solvers may lead to instability. This happens, for example, in the case where $\circled{1} = \circled{2}$. Suppose the initial data at $t = 0 +o(1) >0$ is  \eqref{RSolver2}, which can be considered as a perturbation of a constant initial data, since it is almost constant except in a very small region. One or two misbehaving cars can create such a perturbation. From a mathematical point of view, whether such a perturbation is considered small depends on the norm used. Nevertheless, a solution for this perturbed initial data is \eqref{RSolver2}, leading to stop-and-go pattern, which is far from the constant solution for the constant initial data.
\end{example}

\begin{remark}\label{r:mech4jam2form}
	The above examples reveal a possible mechanism to produce jams through small oscillations in speed in a car platoon traveling in congested zone $v<v_c$. Start at $t=0$ from a uniform car train with constant spacing and speed $(\bar u, \bar v)$ in the scanning zone. Later, some drivers in the platoon drive temporarily and slightly above (or below) $\bar v$: they will make their spacing more (or less) compact. If more drivers drive over the speed $\bar v$, then the car platoon become denser while the overall platoon speed is still $\bar v$, since small speed variations decay as time increases as long as the corresponding states $(u, v)$ are still inside the scanning zone.
	Once the spacing $u$ decreases to $u = (v^D)^{-1}( \bar v)$, over-speeding becomes impossible while under-speeding creates a persistent expanding slower car region upstream. Later, the same mechanism can repeat itself in the slower car region to create an even slower car region that is expanding upstream, causing jam eventually, unless upstream and/or downstream traffic spacing increase sufficiently.

\end{remark}

\section{Conclusions and Discussions}\label{s:conclusions}

In this paper we proposed a new macroscopic model for vehicular traffic flows with hysteresis. It is an extension of LWR model to include hysteretic drivers' behavior,  where  hysteresis loops used are based on the shapes observed experimentally \cite{CLS, SZHW, Treiterer-Myers, Yeo-Skabardonis}. The hysteresis mood parameter for each driver is updated by a governing differential equation.

The new model not only covers the successful part of LWR type of models,
but also produces
wave patterns which are missing in traditional continuum macroscopic traffic models but are observed in real traffic flows.
We showed that the model has solutions corresponding to stop-and-go patterns, car platoons with uniform speed but possibly uneven spacing, stationary shocks, and acceleration shocks. 
All these waves are missing in 
the LWR model. Therefore, these wave patterns are attributable to hysteresis, because hysteresis is the only new factor in the model \eqref{system2h} comparing with the LWR model.
Acceleration shocks have not been pointed out before, but were present in many observations on real traffic flows such as in Fig. 3 of \cite{Kaufmann_et_al2017}.

Examples of initial value problems of the model \eqref{system2h} are studied to gain insights into the formation, growth and decay of traffic jams. 
Through these examples, we gain the following insights: A stop-and-go pattern can form by a few cars under-speeding sufficiently in an otherwise uniform car train,  traveling in either an open lane, or a circular closed loop road. Assume that cars in a platoon traveling in congested zone are more likely to over-speeding than under-speeding with respect to the speed of the leading car of the platoon. Then, at first, the platoon will become more compact while keeping up with the speed $v_+$ of the leading car. Later, when the spacing is reduced to the point $u= (v^D)^{-1}(v_+)$, some under-speeding cars will cause a persistent slower and denser segment moving upstream. If this mechanism repeats itself in the slower and denser segment, traffic jam emerges upstream.
Some examples, studied in this paper, confirm our intuition that a {\em faster moving} downstream flow can eliminate a phantom jam, and that a sufficiently {\em sparse} upstream flow can also eliminate a phantom jam. A small perturbation of the speeds of a few cars in an otherwise uniform car train in congested zone will decay, but the effects on spacing on these cars is persistent. 
Riemann solvers for \eqref{system2h} are often not unique. Although Riemann solvers for ``rational" drivers are given in \S5, there is nothing to stop a driver to over or under speed, which leads to other Riemann solvers to appear in solutions. Some of these ``irrational" Riemann solvers will cause persistent slower traffic upstream.

The above results show that the model \eqref{system2h} and its version in Eulerian coordinates
\eqref{systemInviscidh} are worth further investigation.  Indeed, there are many more problems about \eqref{system2h} or \eqref{systemInviscidh} waiting to be studied. As examples, we raise a few questions here.

\begin{enumerate}[{\em (i)}]

\item One problem to attack is to extend our study from \eqref{system2h} to the system in Eulerian coordinate \eqref{systemInviscidh}. 
Most traffic observations reported in the literature are about traffic flows near fixed bottlenecks. Thus, investigating this problem is a necessary step to take before a comparison of the model's behavior with these observations can be done. 

\item Another question is whether small random speed fluctuation, biased in over-speeding than the leading car, can either cause traffic jam, as suggested in Remark \ref{r:mech4jam2form}, or lead to stop-and-go patterns in a ring road.

\item One more problem to study is the effect of the shape of hysteretic fundamental diagram on traffic patterns and throughput. What if the shape of scanning curves are not concave? What if the scanning curves for acceleration and deceleration are different? From an experimental point of view, it is desirable to have more detailed information about the scanning curves.
\end{enumerate}

\section*{Acknowledgment}
The first author is a member of GNAMPA and acknowledges support from this institution.
The second author would like to thank
NVIDIA Academic GPU Grant for contributing a Titan GPU
enabling him to do massive parallel computing on the GPU.

{\small
\bibliography{refe_th}
\bibliographystyle{abbrv}
}

\end{document}